\documentclass[preprint,12pt]{elsarticle}

\usepackage{amssymb}

\usepackage{amsmath}
\usepackage{amssymb}
\usepackage{enumerate}
\usepackage{amsbsy}
\usepackage{amsfonts}
\usepackage{amsthm,dsfont,color}
\usepackage{mathrsfs}

\newcommand{\nr}{{|\!|}}
\newcommand{\0}{{L^2(\mathbb{R}^N)}}
\newcommand{\s}{H^s(\mathbb{R}^N)}
\newcommand{\stcn}{\xrightarrow[n\to +\infty]{}}
\newcommand{\stck}{\xrightarrow[k\to +\infty]{}}
\newcommand{\stcr}{\xrightarrow[r\to +\infty]{}}
\newcommand{\rij}{{r_{ij}}}
\newcommand{\pij}{{p_{ij}}}
\newcommand{\p}{{\phi}}
\newcommand{\ps}{{\psi}}
\newcommand{\R}{{\mathbb{R}^N}}
\newcommand{\const}{\eta }
\newtheorem{theorem}{Theorem}[section]
\newtheorem{proposition}[theorem]{Proposition}
\newtheorem{lemma}[theorem]{Lemma}

\newtheorem{dfn}[theorem]{Definition}
\begin{document}

\begin{frontmatter}
\title{Orbital stability of standing waves of a class of fractional Schr\"odinger equations with a general Hartree-type
integrand}
 \author[Cho]{Y. Cho}
 \address[Cho]{Department of Mathematics, and Institute of Pure and Applied Mathematics, Chonbuk National University, Jeonju 561-756, South Korea.}
\author[Fall]{M.M. Fall}
\address[Fall]{African Institute for Mathematical Sciences of Senegal, AIMS-Senegal,
KM 2, Route de Joal, B.P. 14 18. Mbour, S\'en\'egal.}
 \author[Hajaiej]{H. Hajaiej}
 \address[Hajaiej]{Department of Mathematics, College of Science, King Saud University
P.O. Box 2455, Riyadh 11451, Saudi Arabia. }
 \author[Kaust]{P.A. Markowich}
 \address[Kaust]{Division of Math \& Computer Sc \& Eng, King Abdullah University of Science and Technology
Thuwal 23955-6900, Saudi Arabia.}
\author[Kaust]{S. Trabelsi}
\begin{abstract}
This article is concerned with the mathematical analysis of a class
of a nonlinear fractional Schr\"odinger equations with a general
Hartree-type integrand. We prove existence and uniqueness of
global-in-time solutions to the associated Cauchy problem. Under
suitable assumptions, we also prove the existence of standing waves
using the method of concentration-compactness by studying the associated constrained minimization
problem. Finally we show the orbital stability of standing waves
which are the minimizers of the associate variational problem.
\end{abstract}
\end{frontmatter}

\section{Introduction}
A partial differential equation is called fractional
when it involves derivatives or integrals of fractional order.
 Various physical phenomena and applications require the use of
 fractional derivatives, for instance quantum mechanics, pseudo-chaotic
 dynamics, dynamics in porous media, kinetic theories of systems with chaotic
  dynamics. The latter application is based on the so called fractional Schr\"odinger
   equation. This equation was derived using the path integral over a kind of L\'evy
   quantum mechanical path approach  by Laskin in Ref.
   \cite{Laskin1,Laskin2,Laskin3}. The mathematical analysis of the fractional
   nonlinear Schr\"odinger equation has been growing continually during the last
   few decades. Many results have been obtained and we refer for instance to
   \cite{chho}
   and references therein.
 \vskip6pt This paper deals with the analysis of the following  Cauchy problem
\begin{align*}\mathscr{S}:\quad\left\{
\begin{array}{l}\label{Cauchy}
i\partial_t \phi +(-\Delta)^s \phi = \left(G(|\phi|) \star
V(|x|)\right)G'(\phi),\\ \\
\phi(t=0,x)=\phi_0.
\end{array}
\right.\end{align*} In the system $\mathscr{S}$, $\phi(t,x)$ is a
complex-valued function on $\mathbb{R}\times \R$ and  $\phi_0$ is a
prescribed initial data in $\s$. The operator $(-\Delta)^s$ denotes
the fractional Laplacian of power $0<s<1$. It is defined as a
pseudo-differential operator
$\mathcal{F}[(-\Delta)^s\,\phi](\xi)=|\xi|^{2s}\,\mathcal{F}[\phi](\xi)$
with $\mathcal{F}$ being the Fourier transform. The symbol $\star$
denotes the convolution operator in $\R$ with the potential
$V(|x|)=|x|^{\beta-N}$ where $\beta>0$ is such that $\beta >N-2s$. The
function $G$ is a differentiable function from
$\mathbb{R}^+\rightarrow \mathbb{R}^+$, $G'(\phi):=\frac{d G}{d
\phi}:= F(|\phi|)\phi$, where $F : \mathbb{R}
\rightarrow \mathbb{R}$.
\vskip6pt
The above Cauchy problem reduces to the massless boson Schr\"odinger
equation in three dimensions when $G(\phi) = |\phi|^2$, $V(|x|)
=|x|^{-1}$ and $s = \frac{1}{2}$. In this case, standing waves of
the system $\mathscr{S}$, i.e. solutions of the form $\phi(t,x) =
u(x)e^{-i\kappa t}$, satisfy the following semilinear partial differential
equation
\begin{equation}\label{(B)}
(-\Delta)^{1/2}u-(|x|^{-1} \ast\;u^2)u + \kappa u = 0.
\end{equation}
The associated variational problem
\begin{align}
\mathcal{I}_\lambda &= \inf\left\{\nr|\xi|^\frac12 \mathcal{F}[{u}](\xi)\|^2_{L^2(\R)} - \int_{\mathbb{R}^N\times \R}\frac{|u(x)|^2|u(y)|^2}{|x-y|}dxdy, \right.\nonumber\\
&\hskip150pt\left.u\in H^{\frac12}(\R),\:\int_{\R} |u(x)|^2\,dx=\lambda \right\rbrace, \label{(M)}
\end{align}
has played a fundamental role
in the mathematical theory of gravitational collapse of boson stars,
\cite{Lieb15}. In Ref. \cite{hajaiej13}, the authors studied the associated variational
problem
\begin{align}
\mathcal{I}^G_\lambda &= \inf\left\{\nr|\xi|^s \mathcal{F}[{u}](\xi)\|^2_{L^2(\R)} -  \int_{\mathbb{R}^N\times \R}
G(u(x))V(|x-y|)G(u(y))dxdy, \right.\nonumber\\
&\hskip170pt\left.u\in H^{s}(\R),\:\int_{\R} |u(x)|^2\,dx=\lambda \right\rbrace, \label{(V)}
\end{align}
 for a
general nonlinearity $G$, a kernel $V(|x|) = |x|^{\beta-N}$ and
dimension $N$, where here and the following
$$
\s:=\{u\in L^2(\R^N),\:\,\nr|\xi|^s \mathcal{F}[{u}](\xi)\|^2_{L^2(\R)}<\infty \}.
$$
 In the critical case $2s = {N-\beta}{}$, they
were able to extend the results of \cite{Lieb15}. Moreover, in the
subcritical $2s > {N-\beta}$, they have also proved the
existence and symmetry of all
minimizers of \eqref{(V)} by using rearrangement techniques. More
precisely, they showed that under suitable assumptions on $G$, one
can always take a radial and radial by decreasing minimizing
sequence of  problem \eqref{(V)}.
\vskip6pt
Another very important issue related to the nonlinear fractional
Schr\"odinger equation $\mathscr{S}$ is the orbital stability of
standing waves. For such an issue, it is essential to show that all the
minimizing sequences are relatively compact in $H^s(\mathbb{R}^N)$.
This is the gist of the breakthrough paper \cite{Lions1}. The line of attach consists of:
\begin{enumerate}
\item Prove the uniqueness of the solutions of $\mathscr{S}$.
\item Prove the conservation of energy and mass of the solutions.
\item Prove the relative compactness of all minimizing sequences of
the problem \eqref{(V)}.
\end{enumerate}
 \vskip6pt Our first result concerns the well-posedness of
the system $\mathscr{S}$. Before  stating it, we need to fix some conditions on $G$. We  assume that $G$
is nonnegative and differentiable such that $G(0)=0$ and for all
$\psi\in\mathbb{R}_+$
\[
\mathcal{A}_0:\exists \,\mu \in \left[\left.2,1+\frac{2s+\beta}{N}
\right)\right.\:\text{s.t.}\:\quad \left\lbrace\begin{array}{ll}&G(\psi)
\leq\const(|\psi|^2 + |\psi|^\mu),\\& \\& |G'(\psi)| \leq\const(|\psi| +
|\psi|^{\mu-1}).\end{array}\right.
\]
We have obtained the following
\begin{theorem}\label{thm1}
Let $N \geq 1,0 < s < 1, 0 < \beta < N, N-2s \le \beta, \phi_0\in H^s(\R)$
and $G$ such that $\mathcal{A}_0$ holds true. Then, there exists a
weak global-in-time solution $\phi(t,x)$ to the system $\mathscr{S}$ such that
\[\phi\in L^\infty(\mathbb{R}\,;\, H^s(\R)) \cap W^{1,\infty}(\mathbb{R}\,;\,
 H^{-s}(\R)).\]
Moreover, if $N = 1$ and $\frac12  < s <1$ or if $N \ge 3$, $\frac{N}{2(N-1)} < s < 1$, $N-s+\frac12 < \beta <
\min(N, \frac{3N}{2}-s-\frac{N}{4s})$ and $\mu$ (in $\mathcal{A}_0$) is such that
\[
\max\left(2, 1 + \frac{2\beta-N}{N-2s}\right) < \mu < 2+
 \frac{N}{N-2s} \frac{2s-1-2N+2\beta}{2s-1+N},
\]
then the solution is unique.
\end{theorem}
The particular case  $\mu=2$ and $2s=N-\beta$ was treated in Ref.
\cite{chho} and for lightness of the proofs, we shall sometimes omit it and  focus
on the case $ \mu \in \left(2,1+\frac{2s+\beta}{N}\right)$. The
proof of the existence part of Theorem \ref{thm1} is based on a
classical contraction argument and the conservation laws associated
to the dynamics of the system $\mathscr{S}$.
The uniqueness part for $N = 1$ of Theorem \ref{thm1} readily follows from the
embedding $H^s \hookrightarrow L^\infty$ for all  $s>\frac12$. The part for $N \ge 3$ is obtained using mixed norms to be defined later
and weighted Strichartz and convolution inequalities, which require  $N \ge 3$.
It would be very interesting to find estimates to handle the uniqueness for $N = 2$. Let us mention that in Ref. \cite{GH} the authors showed  the orbital stability  of standing  waves  in the case of power nonlinearities by assuming energy conservation and time continuity without proving uniqueness, which is an  inescapable and quite hard  step, especially in the fractional setting.
\vskip6pt
As mentioned before, if $\phi(t,x)=e^{i\kappa t}u(x)$ with
$\kappa\in\mathbb{R}$ is a solution of the system $\mathscr{S}$,
then it is called a standing wave solution and $u(x)$ solves the
following bifurcation problem
\[\tilde{\mathscr{S}}:\quad(-\Delta)^s u -\kappa u= \left(G(|u|)
\star V(|x|)\right)G'(u).\] In order to study the existence of a
solution $(\kappa,u)$ to the
 stationary equation $\tilde{\mathscr{S}}$, we use a variational
 method based on the following minimization problem
\begin{equation}\label{optimization-problem}
\mathcal{I}_{\lambda}= \inf\left\lbrace \mathcal{E}(u),\quad u
\in H^s(\mathbb{R}^N),\quad \int_{\mathbb{R}^N} |u(x)|^2 \,dx=\lambda \right \rbrace,
\end{equation}
where $\lambda$ is a positive prescribed number and
\begin{eqnarray*}
\mathcal{E}(u)&=&\frac{1}{2} \nr\nabla_s u\nr^2_{L^2(\R)} -
 \frac{1}{2} \int_{\mathbb{R}^N\times\mathbb{R}^N }G(|u(x)|)
 \,V(|x-y|)\,G(|u(y)|)\,dxdy,\\
 &:=& \frac{1}{2} \nr\nabla_s u\nr^2_{L^2(\R)} -\frac{1}{2}
 \,\mathcal{D}(G(|u|),G(|u|)).
\end{eqnarray*}
The kinetic energy is precisely expressed by the formula for all function $u$ in the Schwarz class
\begin{equation}\label{deffrac}
 \|\nabla_s u\|^2_{L^2(\R)} = C_{N,s} \int_{\mathbb{R}^N\times \R}
 \frac{|u(x)-u(y)|^2}{|x-y|^{N+2s}}dxdy,
\end{equation}
with $C_{N,s}$ being a positive normalization constant. In order to prove the existence of critical points to the functional
$\mathcal{E}$ and thereby  solutions to the problem
$\tilde{\mathscr{S}}$, we  will need some extra grows condition on $G$:
 for all $\psi\in \mathbb{R}_+$
\[
\mathcal{A}_1:\quad
\left\lbrace
\begin{array}{l}
\exists 0<\alpha< 1+\frac{2s+\beta}{N}\:\:s.t.\:\: \forall \psi,\:
0<\psi\ll1,\quad G(\psi)\geq \const\,\psi^\alpha, \\ \\
G(\theta \,\psi) \geq \theta^{1+\frac{2s+\beta}{2N}}\,G(\psi).
\end{array}
\right.
\]
Our next main result is contained in the following
\begin{theorem}\label{thm2}
Let $0<s<1,0<\beta<N, N-\beta\leq 2s$ and $G$ such that
$\mathcal{A}_0$ and $\mathcal{A}_1$ hold true. Then, for all
 $\lambda>0$,  problem \eqref{optimization-problem} has a minimizer
 $u_\lambda\in H^s(\R)$ such that $I_\lambda=\mathcal{E}(u_\lambda)$.
\end{theorem}
In fact we will show that  any minimizing sequence of  problem
\ref{optimization-problem} is --up to suitable translations-- relatively
compact in $H^s(\R)$. The proof of Theorem \ref{thm2} is based on
the concentration-compactness
 method of P-L. Lions \cite{Lions}.
\vskip6pt The last part of the paper deals with the stability of the
standing waves. For that purpose, we introduce the
 following problem
\begin{equation*}
\hat{\mathcal{I}}_{\lambda}= \inf\left\lbrace \mathcal{J}(z),\quad
z \in H^s(\mathbb{R}^N),\quad \int_{\mathbb{R}^N} |z|^2 \,dx=\lambda \right \rbrace,
\end{equation*}
where $z=u+i\,v$ and
\begin{eqnarray*}
\mathcal{J}(z) &=&\frac{1}{2} \nr\nabla_s z\nr^2_{L^2(\R)} -\frac{1}{2}\,
\mathcal{D}(G(|z(x)|),G(|z(x)|)),
\\&=&\frac{1}{2} \nr\nabla_s u\nr^2_{L^2(\R)}+ \frac{1}{2} \nr\nabla_s v\nr^2_{L^2(\R)} -\frac{1}{2}\,\mathcal{D}(G((u^2+v^2)^\frac12),G((u^2+v^2)^\frac12)),\\
\\&:=&\mathcal{J}(u,v).
\end{eqnarray*}
We have obviously $\mathcal{E}(u)=\mathcal{J}(u,0)$. Following Ref. \cite{Lions1}, we introduce the following set
\[\hat{\mathcal{O}}_\lambda=\left\lbrace z \in H^s(\mathbb{R}^N),\quad
 \int_{\mathbb{R}^N} |z|^2 \,dx=\lambda\:\::\:\: \mathcal{J}(z)
  =\hat{\mathcal{I}}_{\lambda}\right\rbrace. \]
The set $\hat{\mathcal{O}}_\lambda$ is the so called orbit of the standing waves of $\mathscr{S}$ with mass
  $\sqrt\lambda$. We define the stability of $\hat{\mathcal{O}}_\lambda$ as follows
\begin{dfn}\label{defstab}
Let $\phi_0\in H^s(\R)$ be an initial data and $\phi(t,x) \in H^s(\R) $ the associated solution of problem $\mathscr{S}$. We say that $\hat{\mathcal{O}}_\lambda$ is $H^s(\R)-$stable with respect to the system $\mathscr S$ if
\begin{itemize}
\item $\hat{\mathcal{O}}_\lambda\neq \varnothing$.
\item For all $ \varepsilon > 0$, there exists $\delta>0$ such
 that for any $\phi_0\in H^s(\R)$ satisfying $\inf_{z \in
\hat{\mathcal O}_\lambda} |\phi_0-z|<\delta$, we have $\inf_{z \in
\hat{\mathcal O}_\lambda}|\phi(t,x)-z|<\epsilon$ for all $t\in \mathbb R$.
\end{itemize}
\end{dfn}
The notion of stability depends then intimately on the well-posedness
 of the Cauchy problem $\mathscr{S}$ and the existence of standing waves. Therefore,   having in hand  Theorems
 \ref{thm1} and \ref{thm2}, we prove the following 
\begin{theorem}\label{thm3}
Let $N \ge 3$, $\frac{N}{2(N-1)} < s < 1$, $N-s+\frac12 < \beta <
\min(N, \frac{3N}{2}-s-\frac{N}{4s})$ and let $G$ satisfying
$\mathcal{A}_0$ and $\mathcal{A}_1$ with $\mu$ (in $\mathcal{A}_0$)
such that
\[
\max\left(2, 1 + \frac{2\beta-N}{N-2s}\right) < \mu < 2+ \frac{N}{N-2s}
 \,\frac{2s-1-2N+2\beta}{2s-1+N}.
\]
Let $\phi_0\in H^s(\R)$ and $\phi(t,x)\in H^s(\R)$ the
associated solution to the problem $\mathscr S$.
Then $\hat{\mathcal{O}}_\lambda$ is $H^s(\R)-$stable with respect
to the system $\mathscr S$.
\end{theorem}
\vskip6pt
The paper is divided into three sections. The first one is dedicated
to the analysis of the dynamics of the system $\mathscr{S}$. More precisely,
in this section we prove Theorem \ref{thm1}. First of all, we prove a local-in-time
existence of solutions. Second we show that under extra assumptions, this solution
is actually unique. Eventually, we use the conservation laws to show
the global-in-time well-posedness. The second section is devoted to the proof
of existence of solution to the problem $\tilde{\mathscr S}$. For that purpose,
 we use the classical concentration compactness method \cite{Lions} to prove
 Theorem \ref{thm2}. The last section is dedicated to the proof of stability
 of standing waves, namely Theorem \ref{thm3}. Here, we use ideas
 and techniques developed in \cite{hajaiej14}.
\vskip6pt From this point onward, $\eta$ will denote variant
universal constants that may change from line to line of
inequalities. When $\eta$ depends on some parameter, we will write
$\eta(\cdot)$ instead of $\eta$. In order to lighten the notation
and the calculation, we shall use $L^p$ and $H^s$ instead of
$L^p(\R)$ and $H^s(\R)$ respectively for real or complex valued
functions. Also, we shall use $\nr\cdot\nr_p$ instead of
$\nr\cdot\nr_{L^p(\R)}$
 for all $p\in [1,\infty]$. The exponent $p'$ will denotes the conjugate
 exponent of $p$, that is $\frac1p+\frac{1}{p'}=1$. For a more
 detailed account about the Sobolev spaces $H^s$, we refer
 the reader to any textbook of functional analysis (see \cite{caz} for instance).

\section{Well-posedness of the system $\mathscr S$}
In this section we consider the local and global well-posedness of
the problem $\mathscr{S}$ and prove Theorem \ref{thm1}. Let us
denote the nonlinear term $[V(|x|)\star G(\phi)] G'(\phi)$ by
$\mathcal N(\phi)$. Since the well-posedness of the case $\mu = 2,
2s = N-\beta$ was treated in \cite{chho}, in this paper we consider
the initial value problem   $\mathscr{S}$ with $\mu \in \left(2, 1 +
\frac{2s+\beta}{N}\right)$. Let $g = G'$, that is,
$\int_0^{|z|}g(\alpha)\,d\alpha = G(z)$, and assume that $g(z) =
\frac{z}{|z|}g(|z|), z \neq 0$, $G(z) \ge 0$. Then, with
$\mathcal{A}_0$, the function $g$ satisfies obviously
\begin{align}\label{cond-g}
|g(z)| + |g'(z)z| \leq C(|z| + |z|^{\mu-1})\;\; \mbox{for all}\;\; z \in \mathbb C.
\end{align}
\subsection{Weak solutions}
We first show existence of weak solutions to $\mathscr S$ in $H^s$.
For this purpose we prove that $\mathcal N$ is Lipschitz map from
$L^{p'}$ to $L^r$ for some $p, r \in \left.\left[2,
\frac{2N}{N-2s}\right)\right.$. Then the rest of the proof is quite straightforward from the Lipschitz map and
well-known regularizing arguments and we refer the readers to the book \cite{caz}.
\begin{proposition}\label{ws1}
Let $N \ge 2$, $0 < s < 1$, $0 < \beta < N$ and $2s \ge N - \beta$. If $g$ satisfies \eqref{cond-g} with $\mu \in \left(2, 1 + \frac{2s+\beta}{N}\right)$. Then there exists a weak solution $\phi$ such that
\begin{align*} &\phi \in L^\infty(-T_{min}, T_{max} ; H^s) \cap W^{1, \infty}(-T_{min}, T_{max} ; H^{-s}),\\ &\\
&\nr\phi(t)\nr_{2} = \nr\phi_0\nr_2,\;\;\mathcal J(\phi(t)) \le \mathcal J(\phi_0).
\end{align*}
 for all $t \in (-T_{min}, T_{max})$, where $(-T_{min}, T_{max})$ is the maximal existence time interval of $\phi$ for given initial data $\phi_0$.
\end{proposition}
\begin{proof}
Let us introduce the following cut-off for the function $g$, $g_1(\alpha) = \chi_{\{0 \le \alpha < 1\}}g(\alpha)$ and $g_2(\alpha) = \chi_{\{\alpha \ge 1\}}g(\alpha)$ and $G_i(z) = \int_0^{|z|}g_i(\alpha)\,d\alpha$ with obvious definition of the Euler function $\chi$. Then, one can writes
\[
\mathcal N(\phi) = \sum_{i,j = 1, 2} \mathcal N_{ij}(\phi)\:\:\text{where}\:\:\mathcal N_{ij}(\phi) = \int_{\R} |x-y|^{-(N-\beta)}G_i(|\phi|)\,dy \,g_j(\phi).
\]
We claim that there exist $p_{ij}, r_{ij} \in \left[\left.2, \frac{2N}{N-2s}\right)\right.$\footnote{If $N =1$ and $\frac12 \le s < 1$, then $\frac{2N}{N-2s}$ is interpreted as $\infty$.} such that
\begin{align}\label{nonlinear est}
\nr\mathcal N_{ij}(\phi)-\mathcal N_{ij}(\psi)\nr_{p'_{ij}} \leq \eta(K) \nr\phi-\psi\nr_{r_{ij}},
\end{align}
for some constant $\const(K)$ with $\eta(K) \leq \eta \,K^{a_{i\!j}}$, $a_{i\!j} > 0$ for all $1\leq i,j\leq 2$, provided $\nr\phi\nr_{H^s} + \nr\psi\nr_{H^s} \le K$.
This implies that $\mathcal N: H^s \to H^{-s}$ is a Lipschitz map on a bounded sets of $H^s$. Indeed, let $\mu_1 = 2$ and $\mu_2 = \mu$. Then we have
\begin{align*}
|\mathcal N_{ij}(\phi)-\mathcal N_{ij}(\psi)| &\le \eta \int_{\R} |x-y|^{-(N-\beta)}(|\phi|^{\mu_i-1} +|\psi|^{\mu_i-1})|\phi-\psi|\,dy |\phi|^{\mu_j-1}\\
&+ \eta \int_{\R} |x-y|^{-(N-\beta)}|\psi|^{\mu_i}\,dy (|\phi|^{\mu_j-2} + |\psi|^{\mu_j-2})|\phi-\psi|.
\end{align*}
By H\"{o}lder's and Hardy-Littlewood-Sobolev inequalities with indices $p_{ij}, r_{ij}$ such that
\begin{align}\label{pijrij}
1-\frac1{p_{ij}} = \frac{\mu_i}{r_{ij}} - \frac\beta{N} + \frac{\mu_j-1}{r_{ij}},\;\;\frac{\mu_i}{r_{ij}} > \frac\beta{N},
\end{align}
we obtain
\begin{align*}
\nr\mathcal N_{ij}(\phi) - \mathcal N_{ij}(\psi)\nr_{p'_{ij}} &\le \eta\,\left[(\nr\phi\nr_{r_{ij}}^{\mu_i-1}+\nr\psi\nr_{r_{ij}}^{\mu_i-1})\nr \phi\nr_{r_{ij}}^{\mu_j-1}\right. \\&\left.+ \nr\psi\nr_{r_{ij}}^{\mu_i}(\nr\phi\nr_{r_{ij}}^{\mu_j-2}+\nr\psi\nr_{r_{ij}}^{\mu_j-2})\right]\nr\phi-\psi\nr_{r_{ij}}.
\end{align*}
Thus if $p_{ij}, r_{ij} \in \left.\left[2, \frac{2N}{N-2s}\right)\right.$, then Sobolev inequality shows \eqref{nonlinear est}.
Now we show that there exist  $\pij, \rij \in [2, \frac{2N}{N-2s})$ such that  the combinations \eqref{pijrij} hold true. If $\pij, \rij$ satisfy \eqref{pijrij}, then they are on the line
\begin{align}\label{line}
\frac1{\rij} = \frac1{\mu_i+\mu_j-1}(1 + \frac\beta{N} - \frac1{\pij}).
\end{align}
Since $\frac1{\mu_i+\mu_j-1}(1+\frac\beta{N}-\frac12) < \frac12$ and $\frac{N-2s}{2N} < \frac1{\mu_i+\mu_j-1}(1+\frac\beta{N}-\frac{N-2s}{2N})$, the line \eqref{line} of $(\frac1{\pij}, \frac1{\rij})$ always passes through the open square $(\frac{N-2s}{2N}, \frac12) \times (\frac{N-2s}{2N}, \frac12)$. We have only to find a pair $(\frac1{\pij}, \frac1{\rij})$ of line \eqref{line} such that $\frac{\mu_i}{\rij} > \frac\beta{N}$. If $\frac{\mu_i}{\rij} > \frac\beta{N}$, then
$$
\frac1{\pij} < 1 - \frac{\mu_j-1}{\mu_i}\frac{\beta}{N}.
$$
So, it suffices to show that
\begin{align}\label{last}
\max\left(\frac1{p_0}, \frac{N-2s}{2N}\right) < 1 - \frac{\mu_j-1}{\mu_i}\frac{\beta}{N},
\end{align}
where $\frac1{p_0}$ is the point of line \eqref{line} when $\frac1{\rij} = \frac12$, that is, $\frac1{p_0} = 1 + \frac\beta{N} - \frac{\mu_i+\mu_j-1}{2}$. In fact, it is an easy matter to show \eqref{last} from the condition $\mu \in \left(2, 1 + \frac{\beta+2s}{N}\right)$ and we leave the proof to the reader. The proof of Proposition \ref{ws1} follows now by a straightforward application of a contraction argument.
\end{proof}
\subsection{Uniqueness}
Since the case $N = 1$ can be treated as in \cite{caz}, we omit the details. When $N \ge 3$, the uniqueness of weak solutions can be shown   by a weighted Strichartz
and convolution estimates. For that purpose, we introduce the
following mixed norm for all $1\le m, \widetilde m < \infty$
\[
\nr h\nr_{L_\rho^m L_\sigma^{\widetilde m}}:=(\int_0^\infty
(\int_{S^{N-1}}|h(\rho\sigma)|^{\widetilde m}\,d\sigma)^\frac
 m{\widetilde m}\,\rho^{n-1} d\rho)^\frac1m.
\]
The case $m = \infty$ or $\widetilde m =\infty$ can be defined is a usual way. Then we have the following.
\begin{proposition}\label{unique}
Let $N \ge 3$, $\frac{N}{2(N-1)} < s < 1$, $N-s+\frac12 < \beta < \min(N, \frac{3N}{2}-s-\frac{N}{4s})$, and
 $g$ such that the condition \eqref{cond-g} holds true with
\[
\max\left(2, 1 + \frac{2\beta-N}{N-2s}\right) < \mu < 2+ \frac{N}{N-2s}\,\frac{2s-1-2N+2\beta}{2s-1+N}.
\]
Then the $H^s$-weak solution to the problem $\mathscr S$ constructed in proposition \ref{ws1} is unique.
\end{proposition}
The dimension restriction $N \ge 3$ is necessary for $\frac{N}{2(N-1)} < s < 1$ and
$N-s+\frac12 < \beta < \frac{3N}{2}-s-\frac{N}{4s}$, which are needed for the exponents appearing in \eqref{exp-cond}.
\begin{proof}
Let $U(t) = e^{it(-\Delta)^s}$,  then the solution $\p$ constructed in Proposition \ref{ws1} satisfies the integral equation
\begin{align}\label{int eqn}
\p(t) = U(t)\varphi - i\int_0^t U(t-t')\mathcal N(\p(t'))\,dt'\;\;\mbox{a.e.}\;t \in (-T_{min}, T_{max}).
\end{align}


Before going further, let us recall the following weighted Strichartz estimate (see for instance Lemma 6.2 of \cite{chho} and Lemma 2 of \cite{chhwoz}).
\begin{lemma}\label{wstr}
Let $N \ge 2$ and $2 \le  q < 4s$. Then, for all $\psi\in L^2$, we have
\[
\nr|x|^{-\delta}  U(t)\psi\nr_{L^q(-t_1, t_2; L_\rho^qL_\sigma^{\widetilde q})} \le \eta\,\nr\psi\nr_{2},
\]
where $\delta = \frac{N+2s}{q}-\frac N2$, $\frac1{\widetilde q} = \frac12 - \frac1{N-1}\left(\frac{2s}{q} - \frac12\right) $ and $\eta$ is independent of $t_1, t_2$.
\end{lemma}
In \cite{chho} it was shown that
\[
\nr|x|^{-\delta} D_\sigma^{\frac{2s}q-\frac12} U(t)\psi\nr_{L^q(-t_1, t_2; L_\rho^qL_\sigma^2)} \le \eta \nr\psi\nr_{2}.
\]
 Lemma \ref{wstr} can be derived by Sobolev embedding on the unit sphere. Here $D_\sigma = \sqrt{1 - \Delta_\sigma}$ where $\Delta_\sigma$ is the Laplace-Beltrami operator on the unit sphere. Now, let us recall the following weighted convolution inequality we shall use in the sequel
\begin{lemma}[Lemma 4.3 of \cite{chonak}]\label{wconv}
Let $r \in [1, \infty]$ and $0 \le \delta \le \gamma < N-1$. If $\frac1r > \frac{\gamma}{N-1}$, then for all $f$ such that $|x|^{-(\gamma-\delta)}f\in L^1$, we have
\[
\nr|x|^\delta (|x|^{-\gamma}\ast f)\nr_{L_\rho^\infty L_\sigma^r} \le \eta\nr|x|^{-(\gamma-\delta)}f\nr_1.
\]
\end{lemma}
Therefore, using Lemma \ref{wstr} one can readily deduce that
\begin{align}\label{w-str-inhomo}
\nr|x|^{-\delta}  \int_0^t U(t-t')f(t')\nr_{L^{q}(-t_1,t_2;
L_\rho^{q} L_\sigma^{\widetilde q})} \le \eta  \nr f\nr _{L^1(-t_1,t_2; L^2)}.
\end{align}
Thus, if we set $f = \mathcal N(\p) - \mathcal N(\ps)$ and $\gamma = N-\beta$.
Then from \eqref{int eqn} we infer
\begin{eqnarray*}
\lefteqn{\nr\p-\ps\nr_{L^\infty(-t_1, t_2; L^2)} + \nr|x|^{-\delta} (\p - \ps)\nr_{L^{q}(-t_1,t_2; L_\rho^{q} L_\sigma^{\widetilde q})}}\\
 &\le& \eta \sum_{i,j = 1}^2\int_{-t_1}^{t_2}\nr\mathcal N_{ij}(\p) - \mathcal N_{ij}(\ps)\nr_{2}\,dt',\\
&\le& \eta \sum_{i,j = 1}^2\int_{-t_1}^{t_2}\nr\int_{\R} |x-y|^{-\gamma}(|\p|^{\mu_i-1} + |\ps|^{\mu_i-1})|\p-\ps|\,dy |\p|^{\mu_j-1}\nr_2\,dt',\\
&+& \eta\sum_{i,j = 1}^2\int_{-t_1}^{t_2}\nr\int_{\R} |x-y|^{-\gamma}|\ps|^{\mu_i}\,dy (|\p|^{\mu_j-2} + |\ps|^{\mu_j-2})|\p-\ps|\nr_2\,dt',\\
&\equiv &\sum_{i,j=1}^2(\mathcal T^1_{ij} + \mathcal T^2_{ij}).
\end{eqnarray*}
We first estimate $\mathcal T^1_{ij}$ using H\"older's and Hardy-Littlewood-Sobolev inequalities. On the one side if $(i,j)=(1,2)$, since $\mu \in \left(1 + \frac{2\beta-N}{N-2s}, 1 + \frac{\beta+2s}{N}\right)$, $0 < \beta < N$ and $2s > \gamma = N-\beta$, we can find $r \in \left[2, \frac{2N}{N-2s}\right]$ such that
\[ \frac\beta{N} = \frac1r + \frac{(\mu-1)(N-2s)}{2N},\;\; \frac1r + \frac12 > \frac{\beta}{N}.
\]
Thus, we can write
\begin{align*}
\mathcal T^1_{12} &\le \eta\int_{-t_1}^{t_2}(\nr\p\nr_r+\nr\ps\nr_r)\nr\p-\ps\nr_2 |\p|_{\frac{2N}{N-2s}}^{\mu-1} \,dt',\\
&\le \eta(t_1+t_2)(\nr\p\nr_{L^\infty(-t_1, t_2; H^s)}^{\mu} + \nr\ps\nr_{L^\infty(-t_1, t_2; H^s)}^{\mu})\nr\p-\ps\nr_{L^\infty(-t_1, t_2; L^2)}.
\end{align*}
On the opposite side, if $(i,j)\neq(1,2)$, then we can choose $r \in \left[2, \frac{2N}{N-2s}\right]$ such that
\[\frac\beta N=\frac{\mu_i-1}{r} + \frac{(\mu_j-1)}r,\;\; \frac{\mu_i-1}r + \frac12 > \frac{\beta}{N}.\]
Such a combination is always possible thanks to our conditions on $\mu, \beta$ and $ s$. Therefore, we get as above
\begin{align*}
\mathcal T^1_{ij} &\le \eta\int_{-t_1}^{t_2}(\nr\p\nr_r^{\mu_i-1}+\nr\ps\nr_r^{\mu_i-1})\nr\p-\ps\nr_2 \nr\p\nr_{\frac{r}{\mu_j-1}}^{\mu_j-1} \,dt',\\
&\le \eta (t_1+t_2)(\nr\p\nr_{L^\infty(-t_1, t_2; H^s)}^{\mu_i+\mu_j-2} + \nr\ps\nr_{L^\infty(-t_1, t_2; H^s)}^{\mu_i+\mu_j-2})\nr\p-\ps\nr_{L^\infty(-t_1, t_2; L^2)}.
\end{align*}
We are kept with the estimates of $\mathcal T^2_{ij}$. If $j = 1$, then we can use  Hardy-Sobolev inequality such that for $0 < q < N$ and $2 \le p < \infty$
\begin{align}\label{h-s}
\nr|x|^{-\frac qp}f\nr_{p} \le \eta\nr f\nr_{\dot H^{\frac N2 - \frac{N-q}{p}}}.
\end{align}
In fact, we have
\begin{align*}
\mathcal T^2_{11} &\le\int_{-t_1}^{t_2}\nr\int_\R |x-y|^{-\gamma}|\ps|^2\,dy\nr_{L_x^\infty}\nr\p-\ps\nr_{2}\,dt',\\
&\le \eta\int_{-t_1}^{t_2}\nr\ps\nr_{\dot H^\frac\gamma2}^2\nr\p-\ps\nr_{2}\,dt',\\
&\le \eta(t_1+t_2)\nr\ps\nr_{L^\infty(-t_1,t_2; H^s)}^2\nr\p-\ps\nr_{L^\infty(-t_1, t_2; L^2)}.
\end{align*}
Since $\frac N2 - \frac\beta\mu \le s$ we also have
\begin{align*}
\mathcal T^2_{21} &\le\int_{-t_1}^{t_2}\nr\int_\R |x-y|^{-\gamma}|\ps|^\mu\,dy\nr_{L_x^\infty}\nr\p-\ps\nr_{2}\,dt',\\
&\le C\int_{-t_1}^{t_2}\nr\ps\nr_{\dot H^{\frac N2 - \frac\beta\mu}}^\mu\nr\p-\ps\nr_{2}\,dt',\\
&\le C(t_1+t_2)\nr\ps\nr_{L^\infty(-t_1,t_2; H^s)}^2\nr\p-\ps\nr_{L^\infty(-t_1, t_2; L^2)}.
\end{align*}
When $j = 2$, we use the weighted convolution inequality (Lemma \ref{wconv}).
The hypothesis on $\beta, \mu$ guarantees the existence of exponents $q, \widetilde q$ and $r$ satisfying the conditions of Lemmas \ref{wstr}, \ref{wconv}
and also the following combination
\begin{align}\label{exp-cond} \frac12 = \frac{(\mu-2)(N-2s)}{2N} + \frac1q = \frac1{r} + \frac{(\mu-2)(N-2s)}{2N} + \frac1{\widetilde q} .\end{align}
Hence, using the Hardy-Sobolev inequality \eqref{h-s} we write
\begin{align*}
\mathcal T^2_{i,2} &\le \int_{-t_1}^{t_2} \nr|x|^\delta\int_\R |x-y|^{-\gamma}|\ps|^{\mu_i}\,dy\nr_{L_\rho^\infty L_\sigma^{r}}\left(\nr\p\nr_\frac{2N}{N-2s}^{\mu-2}+\nr\ps\nr_{\frac{2N}{N-2s}}^{\mu-2}\right)\times \\ &\hskip200pt\times\nr|x|^{-\delta}(\p-\ps)\nr_{L_\rho^q L_\sigma^{\widetilde q}}\,dt',\\
&\le \eta\int_{-t_1}^{t_2} \nr|x|^{(-\gamma-\delta)}|\ps|^{\mu_i}\nr_1\left(\nr|\p\nr_{H^s}^{\mu-2}+\nr\ps\nr_{H^s}^{\mu-2}\right)\nr|x|^{-\delta}(\p-\ps)\nr_{L_\rho^q L_\sigma^{\widetilde q}}\,dt',\\
&\le \eta\int_{-t_1}^{t_2} \nr\ps\nr_{\dot H^{\frac N2 - \frac{\beta+\delta}{\mu_i}}}^{\mu_i}\left(\nr\p\nr_{H^s}^{\mu-2}+\nr\ps\nr_{H^s}^{\mu-2}\right)\nr|x|^{-\delta}(\p-\ps)\nr_{L_\rho^q L_\sigma^{\widetilde q}}\,dt',\\
&\le \eta(t_1+t_2)^{1-\frac1q}\left(\nr\p\nr_{L^\infty(-t_1, t_2; H^s)}^{\mu_i+\mu-2}+\nr\ps\nr_{L^\infty(-t_1, t_2; H^s)}^{\mu_i+\mu-2}\right)\times \\&\hskip200pt\times\nr|x|^{-\delta}(\p-\ps)\nr_{L^q(-t_1, t_2; L_\rho^q L_\sigma^{\widetilde q})}.
\end{align*}
Now, if $(-t_1, t_2) \subset [-T_1, T_2]$ and $\nr\p\nr_{L^\infty(-T_1, T_2; H^s)} + \nr\psi\nr_{L^\infty(-T_1, T_2; H^s)} \le K$, then by combining all the estimates above we infer
 \begin{align*}
&\nr\p-\ps\nr_{L^\infty(-t_1, t_2; L^2)} + \nr|x|^{-\delta} (\p - \psi)\nr_{L^q(-t_1,t_2; L_\rho^q L_\sigma^{\widetilde q})} \le \eta(K^2 + K^{2\mu-2})\times\\ &\hskip40pt \times (t_1+t_2)^{1-\frac1q} \left(\nr\p-\ps\nr_{L^\infty(-t_1, t_2; L^2)} + \nr|x|^{-\delta} (\p - \psi)\nr_{L^q(-t_1,t_2; L_\rho^q L_\sigma^{\widetilde q})}\right).
\end{align*}
Thus, $\p = \psi$ on $[-t_1, t_2]$ for sufficiently small $t_1, t_2$. Let $I = (-a , b)$ be the maximal interval  of $[-T_1, T_2]$ with
\[\nr\p - \psi\nr_{L^\infty(-c, d; L^2)} + \nr|x|^{-\delta} (\p - \psi)\nr_{L^q(-c,d; L_\rho^q L_\sigma^{\widetilde q})}= 0,\:c < a, d < b.
\]
Assume that $a < T_1$ or $b < T_2$. Without loss of generality, we may also assume that $a < T_1$ and $b < T_2$. Then for a small $\varepsilon > 0$ we can find $ a < t_1 < T_1, b < t_2 < T_2$ such that
\begin{align*}
&\nr\p-\psi\nr_{L^\infty(-t_1, t_2; L^2)} + \nr|x|^{-\delta} (\p - \psi)\nr_{L^q(-t_1,t_2; L_\rho^q L_\sigma^{\widetilde q})}\\
&\leq (K^2 + K^{2\mu-2}) (t_1+t_2-a-b)^{1-\frac1q}\times\\ &\hskip50pt \times \left(\nr\p-\psi\nr_{L^\infty(-t_1, t_2; L^2)}+ \nr|x|^{-\delta} (\p - \psi)\nr_{L^q(-t_1,t_2; L_\rho^q L_\sigma^{\widetilde q})}\right),\\
&\le (1-\varepsilon)\left(\nr\p-\psi\nr_{L^\infty(-t_1, t_2; L^2)} + \nr|x|^{-\delta} (\p - \psi)\nr_{L^q(-t_1,t_2; L_\rho^q L_\sigma^{\widetilde q})}\right).
\end{align*}
This contradicts the maximality of $I$. Thus $I = [-T_1, T_2]$. Since $[-T_1, T_2]$ is arbitrarily taken in $(-T_{min}, T_{max})$, we finally get the whole uniqueness and the Proposition \ref{unique} is now proved.
\end{proof}
\subsection{Global well-posedness}
Using the argument of \cite{caz},  one can show
 that the uniqueness
implies actually well-posedness and conservation laws:
\begin{align*}
&\bullet\; \p \in C(-T_{min}, T_{max}; H^s) \cap C^1(-T_{min}, T_{max}; H^{-s}), \\
&\bullet\; \p\;\;\mbox{depends continuously on}\;\; \phi_0\;\;\mbox{in}\;\;H^s,\\
&\bullet\; \nr\p(t)\nr_2 = \nr\phi_0\nr_2\;\;\mbox{ and}\;\;\mathcal
 J(\p(t)) = \mathcal J(\phi_0)\;\; \forall\; t \in (-T_{min}, T_{max}).
\end{align*}
The proofs of these points are standard, we omit them and refer to \cite{caz}. Now we remark that the well-posedness is actually global by establishing a uniform bound on the $H^s$ norm of $\phi(t)$ for all $t\in (-T_{min}, T_{max})$.
\vskip6pt
We first consider the global existence of weak solutions. Suppose $\p$ is a weak solution on $(-T_{min}, T_{max})$ as in Proposition \ref{ws1}. We show that
$\nr\p(t)\nr_{H^s}$ is bounded for all $t \in (-T_{min}, T_{max})$. For this purpose let us introduce the following notation
\begin{equation} \label{def-decomp}
\mathcal{D}(G(|\phi|),G(|\phi|))=\sum_{i,j=1}^2\mathcal{D}_{i,j}(|\phi|),\quad \mathcal{D}_{i,j}(|\phi|):=\mathcal{D}(G_i(|\phi|),G_j(|\phi|)),
\end{equation}
where obviously we set $G_i:=\int_0^{|z|}g_i(\alpha)\,d\alpha$ and recall that the $g_i$ are defined as $g_1(\alpha) = \chi_{\{0 \le \alpha < 1\}}g(\alpha)$ and $g_2(\alpha) = \chi_{\{\alpha \ge 1\}}g(\alpha)$. Using Hardy-Littlewood-Sobolev and the fractional Gagliardo-Nirenberg inequalities and the assumption $\mathcal{A}_0$ we can write the following estimates
\begin{align}
&\mathcal{D}_{1,1} \leq \eta \,\nr u\nr^4_{\frac{4N}{N+\beta}} \leq \eta \nr u\nr^{4-\frac{N-\beta}{s}}_{2}\,\nr u \nr^{\frac{N-\beta}{s}}_{\dot H^s},\label{estimatea} \\
&\nonumber \\
&\mathcal{D}_{2,2} \leq \eta \,\nr u\nr^{2\mu}_{\frac{2N\mu}{N+\beta}}\leq \eta \nr u\nr^{2\mu-\frac{N(\mu-1)-\beta}{s}}_{2}\,\nr u \nr^{\frac{N(\mu-1)-\beta}{s}}_{\dot H^s},\label{estimateb}\\
& \nonumber\\
& \mathcal{D}_{1,2},\,\mathcal{D}_{2,1} \leq \eta \,\nr u\nr^{\mu+2-\frac{N\mu-2\beta}{2s}}_{2} \,\nr u\nr^\frac{N\mu-2\beta}{2s}_{\dot H^s}.\label{estimatec}
\end{align}
Since $N-\beta<2s$, then $0<\frac{N-\beta}{s}<2$ and $4-\frac{N-\beta}{s}>2$. As well since $2\leq \mu < 1+\frac{2s+\beta}{N}$, then $0<\frac{N-\beta}{s}\leq \frac{N(\mu-1)-\beta}{s}<2$ and $ 2< \mu-\frac{N(\mu-1)-\beta}{s}$. Eventually, we have  $0<\frac{N-\beta}{s}\leq \frac{N\mu-2\beta}{2s}$ and $2\leq\mu<\mu+1+\frac{\beta-N}{2s}<\mu+2-\frac{N\mu-2\beta}{2s}$. The estimates above can be summarized as follows with $\mu_1=2$ and $\mu_2=\mu$.
\begin{eqnarray}
\mathcal{D}_{i,j}(|u|)&\leq& \const \int_{\R\times\R}\frac{|u(x)|^{\mu_i}|u(y)|^{\mu_j}}{|x-y|^{N-\beta}}\,dxdy,\nonumber\\
&\leq& \const \,\nr u\nr^{\mu_i+\mu_j-\gamma_{i,j}}_{2} \, \nr u\nr^{\gamma_{i,j}}_{\dot H^{s}} \label{estimate0}
\end{eqnarray}
where
\[ \gamma_{i,j} = \frac{N}{s}\left(1+\frac\beta N\right) - \left(\frac{N}{2s}-1\right)(\mu_i+\mu_j).\]
Thus, we have clearly
\begin{align*}
\frac12 \nr\p\nr_{H^s}^2 &= \frac12 \nr\p\nr_2^2 + \mathcal J(\p) +\mathcal{D}(G(|\phi|),G(|\phi|)), \\
& \le \frac12 \nr\phi_0\nr_2^2 + \mathcal J(\phi_0) + \eta\sum_{i,j = 1,2}\nr\phi_0\nr_2^{\frac{2\gamma_{ij}}{2 - \mu_i-\mu_j+\gamma_{ij}}} + \frac14\nr\p\nr_{H^s}^2.
\end{align*}
Thus
\[ \nr\p\nr_{H^s} \leq \eta \left(\nr\phi_0\nr_{H^s}\right),\quad \text{for all}\:t\in(-T_{min}, T_{max}).\]
Therefore $T_{min} = T_{max} = \infty$. If $s, \beta, \mu$ satisfy the hypothesis of Proposition \ref{unique}, then we get the global well-posedness. Eventually, combining this fact with the Propositions \ref{ws1} and \ref{unique} prove Theorem \ref{thm1}.
\section{Existence of standing waves}
In this section we study the minimization problem $\tilde{\mathscr
S}$. We prove  the existence of a solution to $\tilde{\mathscr S}$
using a variational approach via the concentration-compactness
method of P-L. Lions \cite{Lions}. Indeed, we aim to prove the existence
of critical points to the energy functional
\begin{eqnarray*}
\mathcal{E}(u)&=&\frac{1}{2} \nr\nabla_s u\nr^2_{2} -\frac{1}{2}\,
\mathcal{D}(G(|u|),G(|u|)).
\end{eqnarray*}
In other words, we look for a function $u_\lambda$ such that
\[\mathcal{E}(u_\lambda)=\mathcal I_\lambda= \inf\left\lbrace \mathcal{E}(u),
\quad u \in H^s(\mathbb{R}^N),\quad \int_{\mathbb{R}^N} |u(x)|^2 \,dx=\lambda
\right \rbrace.\]
As noticed in the introduction of this paper, this problem has been studied in various situation depending on the value of $s$ and the conditions on $\beta$ and the integrand $G$ in Ref. \cite{chho,hajaiej13,Lieb15}.
In order to prove the existence of critical points to the functional $\mathcal{E}$, we start with the following claim
\begin{proposition} \label{propc1}
For all $\lambda>0$ and $G$ such that $\mathcal A_0$ and $\mathcal{A}_1$ hold true, we have
\begin{itemize}
\item The functional $\mathcal E \in C^1(H^s,\mathbb R)$ and there exists a constant $\eta>0$ such that
\[ \nr \mathcal{E}'(u)\nr_{H^{-s}} \leq \eta\left(\nr u\nr_{H^s} + \nr u\nr_{H^s}^{\frac{2s+\beta}{N}} \right).\]
\item $-\infty<\mathcal I_\lambda<0$.
\item Each minimizing sequence for the problem $\mathcal I_\lambda$ is bounded in $H^s$.
\end{itemize}
\end{proposition}
\begin{proof}
Let us mention that only assumption $\mathcal{A}_0$ is needed to prove the $C^1$ property of the energy functional $\mathcal{E}$. The proof of this claim is standard and we refer the reader to Ref. \cite{Hichem} for details.  Now, we prove the second assertion. Let $u\in \s$ such that $\nr u\nr_2=\sqrt\lambda$ and assume $\mathcal{A}_0$. Then, on the one hand, thanks to (\ref{estimatea}\,--\ref{estimatec}), it is rather easy to show using Young's inequality that for all $\epsilon_1,\epsilon_2$ and $\epsilon_3$, there exist $C_{\epsilon_1},C_{\epsilon_2},C_{\epsilon_3}>0$ such that
\begin{eqnarray}
\mathcal{D}_{1,1} &\leq& \const \left(\epsilon_1 \,\nr u\nr^2_{\dot H^s}+C_{\epsilon_1} \lambda^{e_1}\right),\quad e_1:=\frac{4s+\beta-N}{2s+\beta-N}.\label{estimated}\\ &&\nonumber\\
\mathcal{D}_{1,2} &\leq& \const \left(\epsilon_2 \,\nr u\nr^2_{\dot H^s}+C_{\epsilon_2} \lambda^{e_2}\right),\quad e_2:=\frac{2s\mu+\beta-N(\mu-1)}{2s+\beta-N(\mu-1)}.\label{estimatee}\\ &&\nonumber\\
\mathcal{D}_{1,2},\,\mathcal{D}_{2,1} &\leq& \const \left(\epsilon_1\nr u\nr^2_{\dot H^s}+C_{\epsilon_3} \lambda^{e_3}\right),\quad e_3:=1+\frac{2s\mu}{4s-N\mu+2\beta}.\label{estimatef}
\end{eqnarray}
Observe that $0<2s+\beta-N< 4s+\beta-N$ so that $e_1>1$. Also, $0<2s+\beta-N(\mu-1)< 2s\mu+\beta-N(\mu-1)$ so that  $e_2>1$. Eventually, $4s-N\mu+2\beta>2s+\beta-N>0$ so that $\frac{2s\mu}{4s-N\mu+2\beta}>0$ and $e_3>1$. Therefore, for sufficiently small  $\epsilon_1,\epsilon_2$ and $\epsilon_3$, one has
\begin{eqnarray*}
\mathcal{E}(u)&\geq& \left(\frac12- \const(\epsilon_1+\epsilon_2+\epsilon_3) \right)\, \nr u\nr^{2}_{H^{s}}-\frac12-\const \left(C_{\epsilon_1} \lambda^{e_1}+C_{\epsilon_2} \lambda^{e_2}+ C_{\epsilon_3} \lambda^{e_3}\right), \\
&\geq &-\frac12 \lambda-\const \left(C_{\epsilon_1} \lambda^{e_1}+C_{\epsilon_2} \lambda^{e_2}+ C_{\epsilon_3} \lambda^{e_3}\right).
\end{eqnarray*}
Thus,  we obtain $\mathcal{I}_\lambda>-\infty$. On the other hand, let us introduce for all $\kappa\in \mathbb{R}$, the rescaled function $u_\kappa=\kappa^\frac12 u(\kappa^{\frac1N}\cdot)$. Obviously, one has $\int_\R |u_\kappa|^2=\lambda$ and using $\mathcal{A}_1$
\begin{eqnarray*}
\mathcal{E}(u_\kappa) \leq \frac12 \kappa^\frac{2s}{N} \int_\R |(-\Delta)^s u(x)|^2 dx -\frac{\kappa^{\alpha -\left(1+\frac\beta N\right)}}{2}\,\mathcal{D}(|u(x)|^\alpha,|u(y)|^\alpha).
\end{eqnarray*}
We have $0< \alpha -\left(1+\frac\beta N\right)< \frac{2s}{N}$, therefore we can take $\kappa$ small enough to get $\mathcal{E}(u_\kappa)<0$. Thus, $\mathcal{I}_\lambda\leq \mathcal{E}(u_\kappa) <0$.
\vskip6pt
We are kept with the proof of the third assertion. Let $(u_n)_{n\in \mathbb{N}}$ be a minimizing sequence for the problem $\mathcal{I_{\lambda}}$.  Therefore, thanks to (\ref{estimated}--\ref{estimatef}), we have  for all $u\in H^s$
\begin{eqnarray*}
\mathcal{D}(G(|u|),G(|u|)) &\leq& \const (\epsilon_1+\epsilon_2+\epsilon_3)\, \nr u\nr^{2}_{\dot H^{s}} +\const \left(C_{\epsilon_1} \lambda^{e_1}+C_{\epsilon_2} \lambda^{e_2}+ C_{\epsilon_3} \lambda^{e_3}\right).
\end{eqnarray*}
Hence
\begin{eqnarray*}
\nr u_n\nr^2_{H^s} &=& 2\,\mathcal{E}(u_n) + \nr u_n\nr^2_{2} + \mathcal{D}(G(|u_n|),G(|u_n|)), \\
&\leq& 2\,\mathcal{I}_\lambda +\lambda + \const (\epsilon_1+\epsilon_2+\epsilon_3)\, \nr u_n\nr^{2}_{ H^{s}} +\const \left(C_{\epsilon_1} \lambda^{e_1}+C_{\epsilon_2} \lambda^{e_2}+ C_{\epsilon_3} \lambda^{e_3}\right).
\end{eqnarray*}
Eventually, we pick $\epsilon_1,\epsilon_2$ and $\epsilon_3$ such that $ \const (\epsilon_1+\epsilon_2+\epsilon_3)<1$, we get immediately that the minimizing sequence $(u_n)_{n\in \mathbb{N}}$ is bounded in $H^s$.
\vskip10pt
\end{proof}
Before going further, let us introduce the so called { L\'evy concentration function}
\[\mathcal{Q}_n(r)=\sup_{y\in \R}\int_{B(y,r)} \,|u_n(x)|^2 dx.\]
It is known that each $\mathcal{Q}_n$ is nondecreasing on $(0,+\infty)$. Also, with the Helly's selection Theorem, the sequence $(\mathcal{Q}_n)_{n\in \mathbb{N}}$ has a subsequence that we still denote $(\mathcal{Q}_n)_{n\in \mathbb{N}}$ by abuse of notation, such that there is a nondecreasing function $\mathcal{Q}(r)$ satisfying
\[\mathcal{Q}_n(r)\stcn \mathcal{Q}(r),\quad\text{for all} \quad r>0.\]
Since $0\leq \mathcal{Q}_n(r)\leq \lambda$, there exists $\beta \in \mathbb{R}$ such that $0\leq \beta \leq \lambda$ such that
\[\mathcal{Q}(r)\stcr \gamma.\]
Briefly speaking, a minimizing sequence $(u_n)_{n\in \mathbb{N}}$ for the problem $\mathcal{I_{\lambda}}$ can only be in one of the following situations:
\begin{itemize}
\item Vanishing, i.e. $\gamma=0$.
\item Dichotomy, i.e. $0<\gamma<\lambda$.
\item Compactness, i.e. $\gamma=\lambda$.
\end{itemize}
In the sequel we shall proceed by elimination and show that
vanishing and dichotomy do not occur. Therefore, compactness holds
true and we are done. We start with the following
\begin{proposition}\label{novanishing}
Let $\lambda>0$ and $(u_n)_{n\in \mathbb{N}}$ be a minimizing
sequence of  problem $\mathcal I_\lambda$ with $G$ such that
 $\mathcal A_0$ and $\mathcal A_1$ hold true. Then $\gamma>0$.
\end{proposition}
The proposition claims then that the situation of vanishing does
not occurs. In the proof of Proposition \ref{novanishing},
we shall use, for all subset of $A\subset \R$, the notation
\[\mathcal{D}\vert_A(G(|u|),G(|u|)):= \int_{A\times A }G(|u(x)|)
\,V(|x-y|)\,G(|u(y)|)\,dxdy.\]
\begin{proof}
Let us first prove that $ \mathcal{D}(G(|u_n|),G(|u_n|))$ is lower
bounded. In other words, we show that  for $n\in \mathbb{N}$ large
enough there exists $\delta >0$ such that
\begin{equation}
\delta <\mathcal{D}(G(|u_n|),G(|u_n|)).\label{propD}
\end{equation}
We argue by contradiction and assume that there exist no such $\delta$, therefore $\liminf_{n\rightarrow+\infty} \mathcal{D}(G(|u_n|),G(|u_n|)) \leq 0$, thus
\begin{eqnarray*}
\mathcal{I}_\lambda= \lim_{n\rightarrow +\infty} \mathcal{E}(u_n)&=&  \lim_{n\rightarrow +\infty}\left(\frac{1}{2} \nr\nabla_s u_n\nr^2_2 -\frac{1}{2}\,\mathcal{D}(G(|u_n|),G(|u_n|))\right)\\
&\geq& - \frac12\lim_{n\rightarrow +\infty}\,\mathcal{D}(G(|u_n|),G(|u_n|))\geq 0.
\end{eqnarray*}
The inequality above is in contradiction with the fact that $\mathcal{I}_\lambda<0$. On the other hand, arguing by contradiction and assuming that the minimizing sequence $(u_n)_{n\in \mathbb{N}}$ vanishes, i.e. assume  that $\gamma=0$. Then there exists a subsequence $(u_{n_k})_{k\in \mathbb{N}}$ of $(u_n)_{n\in \mathbb{N}}$  and a radius $\tilde r>0$ such that
\[ \sup_{y\in \R}\int_{B(y,\tilde r)} \,|u_{n_k}(x)|^2 dx \stck 0.\]
Next, since the sequence $(u_{n_k})_{k\in \mathbb{N}}$ is bounded in $H^s$, then one can find $r_\epsilon>0$ such that
\[ \mathcal{D}\vert_{|x-y|\geq r_\epsilon} ({G(|u_{n_k}|),G(|u_{n_k}|))}\leq \frac\epsilon2. \]
Now, we cover $\R$ by balls of radius $r$ and centers $c_i$ for $i=1,2,\ldots$ such that each point of $\R$ is contained in at most $N+1$ ball. Therefore, there exists $N_\epsilon$ ball and a subsequence $(c_{i_l})_{l=1,\ldots,N_\epsilon}$  such that
\begin{eqnarray*}
\lefteqn{\mathcal{D}\vert_{|x-y|\geq r_\epsilon} ({G(|u_{n_k}|),G(|u_{n_k}|))} \leq \const \sum_{p,q=1}^2 {\mathcal{D}_{p,q}\vert_{|x-y|\geq r_\epsilon}} (|u_{n_k}|),}\\
&\leq&\const \sum_{p,q=1}^2  \sum_{l=1}^\infty  \sum_{i=1}^{N_\epsilon}
\int_{B_x(c_l,r)} \int_{B_y(c_{i_l},r)} \frac{|u_{n_k}(x)|^{\mu_p}|u_{n_k}(y)|^{\mu_q}}{|x-y|^{N-\beta}} dxdy,\\
&\leq&\const \sum_{p,q=1}^2  \sum_{l=1}^\infty  \sum_{i=1}^{N_\epsilon}\nr u_{n_k} \nr_{L^{2}(B_x(c_l,r))}  \nr \int_{B_y(c_{l_i},r)} \frac{|u_{n_k}(y)|^{\mu_q}}{|x-y|^{N-\beta}} dy \,|u_{n_k}|^{\mu_p-1}\nr_{L^2(B_x(c_l,r))}\,,\\
&\leq&N_\epsilon\,\const \left(\sum_{l=1}^\infty \nr u_{n_k} \nr_{L^{2}(B_x(c_l,r))}\right) \nr u_{n_k}\nr_{r} \,\nr u_{n_k}\nr^{\mu-1}_{\frac{2N}{N-2s}}\sup_{y\in \mathbb{R}^N}\, \nr u_{n_k}\nr_{L^2(B(y,r))}\\
&+& N_\epsilon\,\const \sum_{(p,q)\neq (1,2), p,q=1}^2\left(\sum_{l=1}^\infty \nr u_{n_k} \nr_{L^{2}(B_x(c_l,r))}\right)\,\nr u_{n_k}\nr^{\mu_q-1}_{{r_{pq}}} \,\nr u_{n_k}\nr^{\mu_p-1}_{{\frac {r_{pq}}{\mu_p-1}}}\times \\
&&\hskip250pt \times\sup_{y\in \mathbb{R}^N}\, \nr u_{n_k}\nr_{L^2(B(y,r))},
\end{eqnarray*}
where $r$ and $r_{pq}$ are such that
\begin{align*}
&\frac\beta N =\frac1r+(\mu-1)\left(\frac{1}{2}-\frac sN\right),\quad \frac1r+\frac12>\frac\beta N,\\&\\
& \frac{\mu_p-1}{r_{pq}} +\frac{\mu_q-1}{r_{pq}}=\frac\beta N,\quad \frac{\mu_p-1}{r_{pq}} +\frac12>\frac\beta N,\quad (p,q)\neq (1,2).
\end{align*}
Since $1+\frac{2\beta-N}{N-2s}<\mu< 1+ \frac{\beta+2s}{N}, 0<\beta<N$ and $s>\frac{N-\beta}{2}$, it is rather clear that one can find (as in section 2) $r,r_{p,q}\in \left[2,\frac{2N}{N-2s}\right]$. Consequently, we have obviously
\begin{eqnarray*}
\lefteqn{\mathcal{D}\vert_{|x-y|\geq r_\epsilon} ({G(|u_{n_k}|),G(|u_{n_k}|))}\leq (N+1)\,N_\epsilon\,\const \nr u_{n_k} \nr_{{2}}\times} \\ &&\hskip50pt \times\left(\nr u_{n_k}\nr^{\mu}_{H^s}+\nr u_{n_k}\nr^{2}_{H^s}+\nr u_{n_k}\nr^{2(\mu-1)}_{H^s}\right)\left(\sup_{y\in \mathbb{R}^N}\, \int_{B(y,r)} |u_{n_k}|^2\right)^\frac12.\\&&\\
&&\hskip300pt \stcn 0.
\end{eqnarray*}
This shows that if the minimizing sequence $(u_n)_{n\in \mathbb{N}}$ vanishes, then
\[\mathcal{D}(G(|u_n|),G(|u_n|))\stcn0.\]
This is in contradiction with the property \eqref{propD}, namely for $n\in \mathbb{N}$ large enough there exists $\gamma>0$ such that $\mathcal{D}(G(|u_n|),G(|u_n|))>\gamma$. Thus, vanishing does not occurs.
\end{proof}
Now, we show the following
\begin{proposition}\label{cont-sub}
Let $0<\pi<\lambda$ and $G$ such that $\mathcal A_0$ and $\mathcal A_1$ hold true. Then the mapping  $\lambda\mapsto \mathcal{I}_\lambda$ is continuous and $ \mathcal{I}_\lambda<\mathcal I_\pi +\mathcal I_{\lambda-\pi}$.
\end{proposition}
\begin{proof}
Let $\lambda >0$ and $(\lambda_k)_{k\in \mathbb{N}}$ be a sequence of positive numbers such that $\lambda_k\stck \lambda$. Let $\epsilon>0$ and $u\in \s$ such that $\nr u\nr_2=\sqrt\lambda$ and
\[\mathcal{I}_\lambda\leq \mathcal{E}(u)\leq \mathcal{I}_\lambda+\frac\epsilon2.\]
For all $k\in\mathbb{N}$, let $u_k=\sqrt \frac{\lambda_k}{\lambda} u$. Obviously $u_k\in\s$ and $\nr u_k\nr^2_2=\lambda_k$ so that for all $k\in\mathbb{N},\, \mathcal I_{\lambda_k}\leq \mathcal E(u_k)$. Now, we show that $ \mathcal E(u_k)\stck \mathcal E(u)$. First, for all $k\in\mathbb{N}$
\[\nr u_k-u\nr_{\dot H^s} \leq \nr u_k\nr_{\dot H^s}\,\left|1-\sqrt\frac{\lambda_k}{\lambda}\right|.\]
Since any sequence of $\mathcal{I}_\lambda$ is bounded in $\s$ and $\lambda_k\stck\lambda$, then we have obviously $\frac{1}{2} \nr\nabla_s u_k\nr^2_2\stck \frac{1}{2} \nr\nabla_s u\nr^2_2$. Next, following the first assertion of Proposition \ref{propc1}, we have $\mathcal{E}(u) \in C^1(\s,\mathbb{R})$. In particular, one can easily see from the proof of this point that $D(u):=\mathcal{D}(G(|u|),G(|u|))\in C^1(\s,\mathbb{R})$ and
\begin{equation}\left|\mathcal{D}'(u)\right| \leq \const \left(\nr u\nr_{H^s} +\nr u\nr^{\frac{2s+\beta}{N}}_{H^s}\right).\label{derivativeD}\end{equation}
We refer to Ref. \cite{Hichem} for details. Therefore, we have
\begin{eqnarray*}
\left|\mathcal{D}(u_k)-\mathcal{D}(u)\right| &=&\left|\int_0^t
 \frac{d}{dt}\mathcal{D}(tu_k+(1-t)u)dt\right|, \\&\leq& \const
  \sup_{u\in{H^s},\nr u\nr_{H^s}\leq \const} \nr \mathcal{D}'(u)
  \nr_{H^{-s}}\,\nr u_k-u\nr_{H^s},\\
&\leq& \const\,\nr u_k\nr_{H^s}\,\left|1-\sqrt\frac{\lambda_k}{\lambda}\right|\,
\stck0.
\end{eqnarray*}
Thus, we have $\mathcal{E}(u_k)\stck \mathcal E(u)$. Consequently,
we have $\mathcal I_{\lambda_k} \leq \mathcal I_\lambda +\epsilon$
for $k$ large
 enough. Next, for all $k\in\mathbb{N}$, let us choose
 $\tilde{u}_k \in \s$ such that $\nr\tilde{u}_k\nr_2=\sqrt \lambda_k$
 and $\mathcal E(\tilde{u}_k) \leq \mathcal I_{\lambda_k} +\frac1k$.
  Moreover, for all $k\in\mathbb{N}$, we set $\bar u_{k} =
  \sqrt\frac{\lambda}{\lambda_k}\tilde{u}_k$. Obviously,
  since $\bar u_k\in \s$ and $\nr \bar u_k\nr_2^2=\lambda$,
  we have $\mathcal I_{\lambda}\leq \mathcal{E}(\bar{u}_k)$.
   Exactly the same argument as above shows that $\mathcal{E}(\tilde u_k)\stck
    \mathcal
    E(\bar u)$ so that for $k$ large enough, we have $\mathcal
    I_\lambda\leq \mathcal I_{\lambda_k}+\epsilon$.
    Whence, $\lambda\mapsto\mathcal I_\lambda$ is continuous on
    $\mathbb{R}_+^\star$. Eventually, using the energy estimates
    (\ref{estimatea}-\ref{estimatec}) or \eqref{estimate0},
    it is rather easy to show that $ I_\lambda \xrightarrow[\lambda\to 0^+]{} 0$.
    This shows that the mapping $\lambda\mapsto \mathcal{I}_\lambda$ is continuous.

Let us now prove the strict sub--additivity inequality. For that
purpose, we introduce  $u_\theta=\theta^{\kappa}
u(\theta^{\frac\kappa N})$ for all $\kappa> \frac{N}{N+2s}$.
Obviously $u_\kappa\in\s$ and $ \nr
u_\theta\nr_{\0}=\sqrt{\theta\lambda}$. Moreover, using
$\mathcal{A}_1$, we have
\begin{eqnarray*}
\mathcal E(u_\theta) &=&\frac12\int_{\R}|(-\Delta)^{\frac s2} u_\theta|^2dx - \frac12\mathcal{D}(G(|u_\theta|),G(|u_\theta|)), \\
 &\leq& \frac{\theta^{\kappa\left(1+\,\frac{2s}{N} \right)}}{2}\left(\int_{\R}|(-\Delta)^{\frac s2} u|^2dx- \mathcal{D}(G(|u|),G(|u|))\right) = {\theta^{\kappa\left(1+\,\frac{2s}{N} \right)}}\mathcal E(u).
\end{eqnarray*}
Thus, we deduce that $\mathcal{I}_{\theta\lambda} \leq {\theta^{\kappa\left(1+\,\frac{2s}{N} \right)}}\mathcal I_\lambda$ for all $\theta>0$. Now we  let $0<\pi<\lambda$, therefore since $\kappa\left(1+\,\frac{2s}{N} \right)>1$ we have
\begin{eqnarray*}
\mathcal{I}_\lambda \leq \lambda^{\kappa\left(1+\,\frac{2s}{N} \right)}\mathcal I_1&<& \pi^{\kappa\left(1+\,\frac{2s}{N} \right)} \mathcal I_1 +(\lambda-\pi)^{\kappa\left(1+\,\frac{2s}{N} \right)}\mathcal I_1,\\
 &\leq&   \pi^{\kappa\left(1+\,\frac{2s}{N} \right)}  \pi^{-\kappa\left(1+\,\frac{2s}{N} \right)}\mathcal I_\pi+(\lambda-\pi)^{\kappa\left(1+\,\frac{2s}{N} \right)}(\lambda-\pi)^{-\kappa\left(1+\,\frac{2s}{N} \right)}\mathcal I_{\lambda-\pi},\\
 &=& \mathcal I_\pi +\mathcal I_{\lambda-\pi}.
\end{eqnarray*}
In summary, for all  $0<\pi<\lambda$, we have $ \mathcal{I}_\lambda<\mathcal I_\pi+\mathcal I_{\lambda-\pi}$.
\end{proof}
Now, we are able to claim the following
\begin{proposition}\label{nodichotomy}
Let $\lambda>0$ and $(u_n)_{n\in \mathbb{N}}$ be a minimizing sequence of  problem $\mathcal I_\lambda$ with $G$ such that $\mathcal A_0$ and $\mathcal A_1$ hold true. Then dichotomy does not occur for $(u_n)_{n\in \mathbb{N}}$.
\end{proposition}
\begin{proof}
Let us introduce $\xi$ and $\chi$ in $C^\infty$ such that $0\leq \xi,\chi\leq 1$ and
\[
\xi(x)=\left\lbrace\begin{array}{lcl}1&\text{if}& |x|\leq1\\&&\\ 0&\text{if}& |x|\geq2\end{array}\right.,\:\chi(x)=1-\xi(x),\:\nr \nabla\xi\nr_{\infty},\nr\nabla\chi\nr_{\infty}\leq 2.
\]
For all $r>0$, let $\xi_r(\cdot)=\xi(\frac \cdot R)$ and $\chi_r(\cdot)=\chi(\frac \cdot R)$. we will show that dichotomy does not occur by contradicting the fact that for all  $0<\pi<\lambda$, we have $ \mathcal{I}_\lambda<\mathcal I_\pi +\mathcal I_{\lambda-\pi}$ proved in Proposition \ref{cont-sub}.
 Indeed, let $(u_n)_{n\in\mathbb{N}}$ be a minimizing sequence of  problem  $\mathcal{I}_\lambda$ and assume that dichotomy holds. Then, using the construction of \cite{Lions},  there exist
\begin{itemize}
\item $0<\pi<\lambda$,
\item a sequence $(y_n)_{n\in\mathbb{N}}$ of points in $\R$,
\item two increasing sequences of positive real number $(r_{1,n})_{n\in\mathbb{N}}$ and $(r_{2,n})_{n\in\mathbb{N}}$ such that
\[r_{1,n}\stcn +\infty\quad \text{and}\quad \frac{r_{2,n}}{2}-r_{1,n}\stcn +\infty,\]
\end{itemize}
such that the sequences $u_{1,n}=\xi_{r_{1,n}}(\cdot-y_n)u_n$ and $u_{2,n}=\chi_{r_{2,n}}(\cdot-y_n)u_n$ satisfy
\[
\left\lbrace\begin{array}{lll}
&u_n=u_{1,n}\: \text{on}\: B(y_n,{r_{1,n}}),\\&\\
&u_n=u_{2,n}\: \text{on}\: B^c(y_n,{r_{2,n}})=\R\setminus B(y_n,{r_{2,n}}),\\&\\
&\int_\R|u_{1,n}|^2dx\stcn\pi,\:\int_\R|u_{1,n}|^2dx\stcn\lambda-\pi,\\&\\
&\nr u_n-(u_{1,n}+u_{2,n})\nr_{p} \stcn 0,\:\text{for all}\: 2\leq p<\frac{2N}{N-2s},\\&\\
&\nr u_n\nr_{L^p(B(y_n,r_{2,n})\setminus B(y_n,r_{1,n}) )} \stcn 0,\:\text{for all}\: 2\leq p<\frac{2N}{N-2s},\\&\\
&\mathrm{dist}(\mathrm{Supp} (u_{1,n}),\mathrm{Supp} (u_{2,n}) )\stcn +\infty.
\end{array}\right.
\]
We have obviously
\begin{eqnarray*}
\mathcal E(u_n) &=& \mathcal E(u_{1,n}) +\mathcal E(u_{2,n}) +\frac12 \int_\R|(-\Delta)^\frac s2\,u_n|^2 -\frac12 \mathcal{D}(G(|u_{n}|),G(|u_{n}|))dx \\ &-&\frac12 \int_\R\left( |(-\Delta)^\frac s2\,u_{1,n}|^2+|(-\Delta)^\frac s2\,u_{2,n}|^2\right)dx \\
&+&\frac12\left(\mathcal{D}(G(|u_{1,n}|),G(|u_{1,n}|)) +\mathcal{D}(G(|u_{2,n}|),G(|u_{2,n}|))  \right).
\end{eqnarray*}
Now we show the existence of $\epsilon>0$ such that for sufficiently large radius $r_{1,n}$ and $r_{1,n}$ we have
\begin{align}
\frac12 \int_\R\left( |(-\Delta)^\frac s2\,u_n|^2-|(-\Delta)^\frac s2\,u_{1,n}|^2-|(-\Delta)^\frac s2\,u_{2,n}|^2\right)dx\geq -\const \epsilon.\label{part1estimate}
\end{align}
Firs of all, it is rather easy to show that by construction of the sequences $u_{i,n}$ for $i=1,2$, we have
\begin{align*}
& \int_\R\left( |(-\Delta)^\frac s2\,u_n|^2-|(-\Delta)^\frac s2\,u_{1,n}|^2-|(-\Delta)^\frac s2\,u_{2,n}|^2\right)dx \\
&\hskip70pt\geq -\,\int_{\R\times\R} \frac{ |\xi_{r_{1,n}}(x-y_n)-\xi_{r_{1,n}}(y-y_n)|^2|u_n(x)|^2}{|x-y|^{N+2s}} dxdy\\
& \hskip85pt-\,\int_{\R\times\R} \frac{ |\chi_{r_{2,n}}(x-y_n)-\chi_{r_{2,n}}(y-y_n)|^2|u_n(x)|^2}{|x-y|^{N+2s}} dxdy.
\end{align*}
Indeed, the estimate above is justified using the definition \eqref{deffrac} combined with the following basic fact for $u_{1,n}$
\begin{align*}
|u_{1,n}(x)-u_{1,n}(y)|^2&=|\xi_{r_{1,n}}(x-y_n)u_n(x)-\xi_{r_{1,n}}(y-y_n)u_n(y)|^2\\
&\leq\frac12 |\xi_{r_{1,n}}(x-y_n)-\xi_{r_{1,n}}(y-y_n)|^2\left(|u_{1,n}(x)|^2+|u_{1,n}(y)|^2\right)\\
&+ \frac12\left( |\xi_{r_{1,n}}(x-y_n)|^2+|\xi_{r_{1,n}}(y-y_n)|^2\right) |u_{1,n}(x)-u_{1,n}(y)|^2.
\end{align*}
and equivalently for $u_{2,n}$
\begin{align*}
|u_{2,n}(x)-u_{2,n}(y)|^2&=|\chi_{r_{2,n}}(x-y_n)u_n(x)-\chi_{r_{2,n}}(y-y_n)u_n(y)|^2\\
&\leq\frac12 |\chi_{r_{2,n}}(x-y_n)-\chi_{r_{2,n}}(y-y_n)|^2\left(|u_{2,n}(x)|^2+|u_{2,n}(y)|^2\right)\\
&+ \frac12\left( |\chi_{r_{2,n}}(x-y_n)|^2+|\chi_{r_{2,n}}(y-y_n)|^2\right) |u_{2,n}(x)-u_{2,n}(y)|^2.
\end{align*}
In order to show \eqref{part1estimate}, it suffices to show that there exist $\epsilon>0$ such that for large radius $r_{1,n}$ and $r_{2,n}$, we have
\begin{align*}
&\int_{\R\times\R} \frac{ |\xi_{r_{1,n}}(x-y_n)-\xi_{r_{1,n}}(y-y_n)|^2|u_n(x)|^2}{|x-y|^{N+2s}} dxdy\leq \eta\epsilon,\\
&\int_{\R\times\R} \frac{ |\chi_{r_{2,n}}(x-y_n)-\chi_{r_{2,n}}(y-y_n)|^2|u_n(x)|^2}{|x-y|^{N+2s}} dxdy\leq\eta\epsilon.
\end{align*}
We prove the first assertion and the second one follows equivalently. Indeed, we split the sum in two part as follows
\begin{align*}
&\int_{\R\times\R} \frac{ |\xi_{r_{1,n}}(x-y_n)-\xi_{r_{1,n}}(y-y_n)|^2|u_n(x)|^2}{|x-y|^{N+2s}}dxdy\\
&\hskip30pt =\int_{|x-y|\leq r_{1,n}} \frac{ |\xi_{r_{1,n}}(x-y_n)-\xi_{r_{1,n}}(y-y_n)|^2|u_n(x)|^2}{|x-y|^{N+2s}}dxdy\\
&\hskip30pt +\int_{|x-y|> r_{1,n}} \frac{ |\xi_{r_{1,n}}(x-y_n)-\xi_{r_{1,n}}(y-y_n)|^2|u_n(x)|^2}{|x-y|^{N+2s}}dxdy:=\mathcal{T}_1+\mathcal T_2
\end{align*}
Now, we write
\begin{align*}
\mathcal T_1&\leq {r^{-2}_{1,n}}\int_{|x-y|\leq r_{1,n}} \frac{ |u_n(x)|^2}{|x-y|^{N+2s-2}}dxdy\\
&\leq {r^{-2}_{1,n}} \,\int_{\R}|u_n(x)|^2 dx \int_{|x|\leq r_{1,n}} \frac1{|x|^{N+2s-2}}dx \leq \eta\,r^{-2s}_{1,n}\,\int_{\R}|u_n(x)|^2 dx.
\end{align*}
Moreover,
\begin{align*}
\mathcal T_2&\leq r^{-s}_{1,n}\int_{|x-y|> r_{1,n}} \frac{ |\xi_{r_{1,n}}(x-y_n)-\xi_{r_{1,n}}(y-y_n)|^2|u_n(x)|^2}{|x-y|^{N+s}}dxdy \\
&\leq \eta \,r^{-s}_{1,n}\, \int_{\R}|u_n(x)|^2 dx\int_{|x-y|>r_{1,n}} \frac{1}{|x-y|^{N+s}}dy \leq \eta \,r^{-s}_{1,n}\, \int_{\R}|u_n(x)|^2 dx.
\end{align*}
Eventually summing up $\mathcal T_1$ and $\mathcal T_2$ and use the same argument in order to handle the term
$\int_{\R\times\R} \frac{ |\chi_{r_{2,n}}(x-y_n)-\chi_{r_{2,n}}(y-y_n)|^2|u_n(x)|^2}{|x-y|^{N+2s}} dxdy$, one ends with
\begin{align*}
& \int_\R\left( |(-\Delta)^\frac s2\,u_n|^2-|(-\Delta)^\frac s2\,u_{1,n}|^2-|(-\Delta)^\frac s2\,u_{2,n}|^2\right)dx\, \\&\hskip100pt \geq -\eta \,(r^{-2s}_{1,n}+\,r^{-s}_{1,n}+r^{-2s}_{2,n}+\,r^{-s}_{2,n})\, \int_{\R}|u_n(x)|^2 dx.
\end{align*}
The estimate \eqref{part1estimate} follows for $r_{1,n}$ and $r_{2,n}$ large enough.
Next, observe that $|u_n-u_{1,n}-u_{2,n}|\leq 3\, \mathds{1}_{(B(y_n,r_{2,n})\setminus B(y_n,r_{1,n}) )} $ where $\mathds{1}_{(B(y_n,r_{2,n})\setminus B(y_n,r_{1,n}) )}$ denotes the characteristic function of $B(y_n,r_{2,n})\setminus B(y_n,r_{1,n}) $. Now, we have
\begin{eqnarray*}
\lefteqn{\left|\mathcal{D}(G(|u_n|),G(|u_n|))-
\mathcal{D}(G(|v_n|),G(|v_n|))-\mathcal{D}(G(|w_n|),G(|w_n|))\right|}\\&&\\
&\leq & \int_{B(y_n,2r)\setminus \bar B(y_n,2r)} \left(\left|\frac{G(|u_n|)
G(|u_n|)}{|x-y|^{N-\beta}}\right| +\left|\frac{G(|v_n|)G(|v_n|)}{|x-y|^{N-\beta}}
\right| \right.\\  && \hskip200pt+\left.\left|\frac{G(|w_n|)G(|w_n|)}{|x-y|^{N-\beta}}
\right|  \right) dxdy, \\
&\leq& \const \left(\nr
u\nr^{4-\frac{N-\beta}{s}}_{L^2(B(y_n,r_{2,n}) \setminus
B(y_n,r_{1,n}) )}\,\nr u \nr^{\frac{N-\beta}{s}}_{H^s}+\nr
u\nr^{2\mu-\frac{N(\mu-1)-\beta}{s}}_{L^2(B(y_n,r_{2,n})\setminus
B(y_n,r_{1,n}) )}\,\nr u \nr^{\frac{N(\mu-1)-\beta}{s}}_{H^s}
\right)\\&+& \const \,\nr
u\nr^{\mu+2-\frac{N\mu-2\beta}{2s}}_{L^2(B(y_n,r_{2,n}) \setminus
B(y_n,r_{1,n}) )} \,\nr u\nr^\frac{N\mu-2\beta}{2s}_{H^s}  \stcn 0.
\end{eqnarray*}
where we used the estimates (\ref{estimatea}--\ref{estimatec}).
Thus, for $r_{2,n}$ and $r_{1,n}$ large enough we have
\begin{equation}
-\frac12\left(\mathcal{D}(G(|u_n|),G(|u_n|))-\mathcal{D}(G(|v_n|),G(|v_n|))-\mathcal{D}(G(|w_n|),G(|w_n|))\right) \geq -\const \epsilon \label{part2estimate}.
\end{equation}
Summing up \eqref{part1estimate} and \eqref{part2estimate}, we end
up  for large $r_{1,n}$ and $r_{2,n}$ with
\begin{equation}\mathcal E(u_n) - \mathcal E(u_{1,n}) -\mathcal
E(u_{2,n})\geq -\const \epsilon.
\label{energyestimate}\end{equation} Since we have
$\int_\R|u_{1,n}|^2dx\stcn\pi$ and
$\int_\R|u_{1,n}|^2dx\stcn\lambda-\pi$, there exist two positive
real sequences $(\mu_{1,n})_{n\in\mathbb{N}}$ and
$(\mu_{2,n})_{n\in\mathbb{N}}$ such that $|\mu_{1,n}-1|,
|\mu_{2,n}-1|<\epsilon$ and
\[\int_\R |\mu_{1,n}u_{1,n}|^2 dx =\pi,\quad  \int_\R |\mu_{2,n}u_{2,n}|^2
dx = \lambda-\pi,\]
so that
\begin{align*} &\mathcal I_{\pi} \leq \mathcal E(\mu_{1,n}u_{1,n})\leq
\mathcal E(u_{1,n})+\frac{\const \epsilon}{2},\\
& \mathcal I_{\lambda-\pi} \leq \mathcal E(\mu_{2,n}u_{2,n}) \leq
 \mathcal E(u_{2,n}) +\frac{\const \epsilon}{2}.
\end{align*}
Thus, with \eqref{energyestimate}, we have and the continuity of the
mapping $\lambda\mapsto \mathcal I_\lambda$ for all $\lambda >0$, we
have
\[ \mathcal I_{\pi}+ \mathcal I_{\lambda-\pi}-3\const \epsilon
\leq\mathcal E(u_{1,n})+\mathcal E(u_{2,n}) -\const \epsilon \leq
\mathcal{E}(u_n)\stcn \mathcal I_\lambda.\] In summary, we proved
that for all $0<\pi<\lambda$, we have $\mathcal I_{\pi}+ \mathcal
I_{\lambda-\pi}\leq \mathcal I_{\lambda}$ contradicting the strict
sub--additivity inequality proved above. Then, the dichotomy does
not occur.
\end{proof}
Now, we finish the proof of Theorem \ref{thm2}. Since vanishing and dichotomy do not occur for any minimizing sequence $(u_n)_{n\in\mathbb{N}}$ for the problem $\mathcal I_\lambda$, then the compactness certainly occurs. Following the concentration-compactness principle \cite{Lions}, we know that every minimizing sequence $(u_n)_{n\in\mathbb{N}}$ of $\mathcal I_\lambda$ satisfies (up to extraction if necessary)
\[\lim_{r\to+\infty} \lim_{n\to+\infty}\sup_{y\in\R}\int_{B(y,r)}
 |u_n(x)|^2dx=\lambda. \]
That is, for all $\epsilon>0$, there exist $r_\epsilon>0$ and
$n_\epsilon\in\mathbb{N}^\star$ and $\{y_n\}\subset \R$ such that
for all $r>r_\epsilon$ and $n\geq n_\epsilon$, we have
\[\int_{B(y_n,r)} |u_n(x)|^2dx=\lambda-\epsilon\]
Now, let $w_n=u_n(x+y_n)$, we have obviously that $\nr w_n\nr_{H^s}=\nr u_n\nr_{H^s}$
is bounded in $\s$, therefore $(w_n)_{n\in\mathbb{N}}$ (up to extraction if necessary)
converges weakly to $w$ in $\s$. In particular $(w_n)_{n\in\mathbb{N}}$
 converges weakly to $w$ in $\0$ and $\nr w_n\nr_{2}=\sqrt\lambda$.
 Now, let $\tilde r_\epsilon>r_\epsilon$ such that $\nr
 w\nr_{L^2(B^c(0,\tilde r_\epsilon))}<\frac\epsilon2$. Thus,
 there exists $\tilde n_\epsilon\in \mathbb{N}^\star,\:\tilde n_\epsilon >n_\epsilon$
  such that for all $n\geq \tilde n_\epsilon$, we have
  $\nr w_n-w\nr_{L^2(B(0,\tilde r_\epsilon))}<\frac\epsilon 2$.
  Therefore, with  the triangle inequality, we have
\begin{eqnarray*}
\nr w\nr_{2} &\geq& \nr u_n\nr_{2} - \nr w_n-w\nr _{L^2(B(0,\tilde
r_\epsilon))} - \nr w_n-w\nr _{L^2(B^c(0,\tilde r_\epsilon))},\\
&\geq& \nr u_n\nr_{L^2(B(y_n,\tilde r_\epsilon))}- \nr
w_n-w\nr _{L^2(B(0,\tilde r_\epsilon))} - \nr w\nr _{L^2(B^c(0,\tilde r_\epsilon))}
 \geq \sqrt{\lambda-\epsilon}-\epsilon.
\end{eqnarray*}
Passing to the limit we get $\nr w\nr_{2}\geq \sqrt\lambda$.
Since the $L^2$ is lower semi continuous, we obtain that
$\nr w\nr_{2}\leq \liminf_{n\to+\infty}\nr w_n\nr_{2}=\sqrt\lambda$.
 Eventually, we get $\nr w\nr_{2}=\sqrt\lambda$,
 therefore the sequence $(w_n)_{n\in\mathbb{N}}$ converges strongly in $\0$
  to $w$.
\vskip6pt
Also, we have
\begin{eqnarray*}
\lefteqn{\left|\mathcal D(G(|w_n|),G(|w_n|))-D(G(|w|),G(|w|))\right|}\\
 && \hskip70pt\leq\left|\int_0^t \frac{d}{dt}\mathcal D(tG(|w_n|)
 +(1-t)G(|w|))\,dt\right|,\\
 &&\\
&&\hskip70pt\leq\const \sup_{u\in{H^s},\nr u\nr_{H^s}\leq \const}
\nr \mathcal{D}'(u)\nr_{H^{-s}}\,\nr w_n-w\nr_{H^s},\\
&&\\
&&\hskip70pt \leq\const\,\nr w_n-w\nr_{2} +
 \const\nr w_n-w\nr_{\frac{2s+\beta}{N}}\stcn 0.
\end{eqnarray*}
In the last line we used \ref{derivativeD}-kind inequality and again
we refer to \cite{Hichem} for a proof. Using the lower semi-continuity
 of the $-s$ norm, we have $\nr w\nr_{H^s}\leq \liminf_{n\to+\infty}
 \nr w_n\nr_{H^s}$. Summing up, we get clearly
\[\mathcal I_\lambda\leq \mathcal{E}(w)\leq \liminf_{n\to+\infty}\mathcal
 E(w_n)=\mathcal I_\lambda.\]
This shows that $w$ is a minimizer of $\mathcal I_\lambda$ and $w_n\stcn w$
in $\s$. Theorem \ref{thm1} is now proved.
\section{Stability of standing waves}
In this section, we prove the orbital stability of standing waves in the sense
 of Definition \ref{defstab}. That is we prove Theorem \ref{thm3}.

\vskip6pt
We argue par contradiction. Assume that $\hat{\mathcal O}_\lambda$ is not stable,
 then either $\hat{\mathcal O}_\lambda$ is empty or there exist $w \in
 \hat{\mathcal O}_\lambda$ and a sequence $\phi^n_0 \in
H^s$ such that $\nr\phi^n_0 -w\nr_{H^s} \stcn 0$ as $n \rightarrow\infty$ but
\begin{equation}\label{tocontradict}\displaystyle{\inf_{z \in
\hat{\mathcal O}_\lambda}}\nr\phi^n(t_n,.)-z\nr_{H^s} \geq \varepsilon,
\end{equation}
for some sequence $t_n \subset \mathbb{R}$, where $\phi^n(t_n,.)$ is the solution
of the Cauchy problem $\mathscr S$ corresponding to the initial condition $\phi^n_0$.
\vskip6pt
Now let $w_n = \phi^n(t_n,.)$, since ${\mathcal J}(w) = \hat{\mathcal I}_\lambda$,
it follows from the continuity of the $L^2$ norm and $\mathcal J$ in $H^s$ that
$\nr\phi^n_0\nr_2\stcn \sqrt\lambda$ and $\mathcal{J}(w_n) = \mathcal{J}(\phi^n_0)
 =\hat{\mathcal I}_\lambda$. With the conservation of mass and energy associated
 with the dynamics of the system $\mathscr S$, we deduce that
\begin{align*}
&\nr w_n\nr_2=\nr\phi^n_0\nr_2 \stcn\sqrt\lambda\quad\text{and}\quad \mathcal J(w_n)=\mathcal J(\phi_0^n)\stcn \hat{\mathcal I}_\lambda.
\end{align*}
Therefore if $(w_n)_{n\in\mathbb N}$ has a subsequence converging to an element $w\in H^s$: $\nr w\nr_2 = \sqrt\lambda$ and $\mathcal{J}(w) = \hat{\mathcal I}_c$. This  shows that
$w \in \hat{\mathcal O}_\lambda$, but

\[\inf_{z \in \hat{\mathcal O}_\lambda}\nr\phi^n(t_n,.) -z
 \nr_{H^s} \leq \nr w_n-w\nr_{H^s}\]
contradicting \eqref{tocontradict}.
\vskip6pt

In summary, to show the orbital stability of $\hat{\mathcal O}_\lambda$,
one has to prove that $\hat{\mathcal O}_\lambda$ is not empty and that any sequence
$(w_n)_{n\in\mathbb{N}} \subset H^s$ such that
\begin{equation}\label{eqq2}
\nr w_n\nr_2 \stcn \sqrt\lambda\quad \mbox{ and } \quad
\mathcal{J}(w_n) \stcn \hat{\mathcal I}_\lambda,
\end{equation}
is relatively compact in $H^s$ (up to a translation).

\vskip6pt From now on, we consider a sequence $(w_n)_{n\in\mathbb
N}$ satisfying \eqref{eqq2}. Our aim is to prove that it admits a
convergent subsequence to an element $w \in H^s$. \vskip6pt If
$(w_n)_{n\in \mathbb N} \subset H^s$, it is easy to see that
\[(|w_n|)_{n\in\mathbb N} \subset H^s\,;\quad w_n = (u_n,v_n).\]
Thanks to  $\mathcal{A}_0$, we have that $(w_n)_{n\in \mathbb N}$ is
bounded in $H^s$ and hence by passing to a subsequence, there exists
$w = (u,v) \in H^s$ such that
\begin{equation}\label{eqq3}
\left\{ \begin{array}{l}
u_n \mbox{ converges weakly to } u \mbox{ in }
H^s,\\ \\
v_n \mbox{ converges weakly to } v \mbox{ in }
H^s,\\\\
\mbox{ the limit when}\: n\: \mbox{goes to } +\infty \:\mbox{of}\: \nr\nabla_s u_n\nr_2 + \nr\nabla_s v_n\nr_2 \:\mbox{ exists }.
\end{array}\right.\end{equation}
Now, a straightforward calculation shows that
\begin{equation} \label{eqq4}
\mathcal{J}(w_n) - \mathcal E(|w_n|) = \frac{1}{2} \nr\nabla_s w_n\nr^2_2 - \frac{1}{2}
\nr\nabla_s|w_n|\nr^2_2\geq 0.
\end{equation}
Thus we have
\begin{equation} \label{eqq5}
\hat{\mathcal I} = \lim_{n\to+\infty} \mathcal{J}(w_n) \geq \limsup_{n\to+\infty}
\mathcal{E}(|w_n|).
\end{equation}
But
\begin{equation} \label{eqq6}\nr |w_n|\nr^2_2 = \nr w_n\nr^2_2 = \lambda_n
\stcn \lambda.
\end{equation}
By the continuity of the mapping $\lambda\mapsto \mathcal I_\lambda$ (see Proposition \ref{cont-sub}), we obtain
\begin{equation}\lim_{n\to+\infty} \mathcal{J}(w_n) \geq \liminf_{n\to+\infty} \mathcal I_{\lambda_n} = \mathcal I_\lambda \geq \hat{\mathcal{I}}_\lambda.\label{eqq7}\end{equation}
Hence
\[\lim_{n\rightarrow +\infty} \mathcal{J}(w_n) = \lim_{n\rightarrow+
\infty} \mathcal E(|w_n|) = \mathcal I_\lambda = \hat{\mathcal I}_\lambda.\]
The properties (\ref{eqq3}) and the inequalities (\ref{eqq4}) and
\eqref{eqq7} imply that
\begin{equation}\lim_{n\rightarrow +\infty} \nr\nabla_s u_n\nr^2_2
- \nr\nabla_s v_n\nr^2_2 -
\nr\nabla_s(u^2_n + v_n^2)^{1/2}\nr^2_2 = 0, \label{eqq8}\end{equation} which is
equivalent to say that
\begin{equation}
\lim_{n\rightarrow +\infty} \nr\nabla_s w_n\nr^2 = \lim_{n\rightarrow +\infty}
\nr\nabla_s |w_n|\nr^2_2. \label{eqq9}
\end{equation}
The convergence \eqref{eqq6}, the inequality \eqref{eqq7}  and Theorem
\ref{thm2} imply that $|w_n|$ is relatively compact in $H^s$ (up to a
 translation). Therefore, there exists $\varphi \in H^s$ such that
\[(u^2_n + v^2_n)^{1/2} \rightarrow \varphi \mbox{ in} \:\:H^s\:\:\mbox{ and }
\:\: \nr\varphi\nr_2 = \sqrt\lambda \:\:\mbox{with}\:\: \mathcal E(\varphi)
= I_\lambda .
\]
Let us prove that $\varphi = |w| = (u^2+v^2)^{1/2}$. Using \eqref{eqq3}, it follows that $u_n \stcn u$ and $v_n \stcn v$ in $L^2(B(0,R))$
\begin{align*}
&|(u^2_n + v^2_n)^{1/2} - (u^2 + v^2)^{1/2}| \leq |u_n-u|^2 + |v_n-v|^2,\\&\\
&(u^2_n + v^2_n)^{1/2} \stcn(u^2 + v^2)^{1/2}\quad \mbox{ in }
L^2(B(0,R)).
\end{align*}
Thus we certainly have that $(u^2+v^2)^{1/2} = |w| =\varphi$. On the other hand $\nr|w_n|\nr_2 = \nr w_n\nr_2 \stcn \sqrt\lambda = \nr w\nr_2=\nr|w|\nr_2$. Therefore, we are done if we prove that $\lim_{n\to \infty}\nr\nabla_s w_n\nr^2_2 = \nr\nabla_s w\nr^2_2$. From \eqref{eqq9}, we have that $\lim_{n\to +\infty}\nr\nabla_s w_n\nr^2_2 = \lim_{n\to +\infty} |\nabla_s |w_n|\nr^2_2$ and $\lim_{n\to +\infty} \nr\nabla_s|w_n|\nr^2_2 = \nr\nabla_s|w|\nr^2_2$. Hence by the lower semi-continuity of $\nr \nabla_s\cdot\nr_2$, we obtain
\begin{equation}\label{eqq11}\nr\nabla_s w\nr^2_2 \leq \lim_{n\to+\infty} \nr\nabla_s|w_n|\nr^2_2 = \nr\nabla_s |w|\nr^2_2.\end{equation}
Eventually, using \eqref{eqq4}, it follows that
\[\nr\nabla_s w\nr^2_2 \geq \nr\nabla_s|w|\nr^2_2.\]
Since by \eqref{eqq3}, we know that  $w_n$ converges weakly to $w$ in $H^s$, it follows that $w_n \stcn w$ in $H^s$, which completes the proof.\\
\vskip6pt
Now, we turn to the characterization of the Orbit $\hat{\mathcal O}_\lambda$.
 We show the following
\begin{proposition}
With the same assumptions of Theorem \ref{thm3}, we have
\[\hat{\mathcal O}_\lambda = \left\{e^{i\sigma}w(.+y),\quad\sigma
\in \mathbb{R}, y \in \mathbb{R}^N\right\},\]
$w$ is a minimizer of \eqref{optimization-problem}.
\end{proposition}
\begin{proof}
Let $z =(u,v) \in \hat{\mathcal O}_\lambda$ and set $\varphi
= (u^2+v^2)^{1/2}$. By the previous section, we know that $\mathcal
E(\varphi) = I_\lambda$, thus $\varphi$ satisfies the partial differential equation :
\begin{equation}(-\Delta)^s \varphi  +\kappa \varphi  =  
 V\star G(|\varphi| ) G'(\varphi ), \label{eq12}
\end{equation}
where $\kappa$ is a Lagrange multiplier. Furthermore the equality
$\|\nabla_s w\|_2 = \|\nabla_s|w|\|_2$ implies that
\begin{equation} u(x)v(y) - v(x) u(y) = 0.\label{eqq13}\end{equation}
By Proposition \ref{prop:reg}, it is plain that   $\varphi \in C(\mathbb{R}^N)$ and $ V\star G(|\varphi|)\in C(\mathbb{R}^N)$.
We can write $(-\Delta)^s \varphi  +\kappa \varphi  =  
 V\star G(|\varphi|)  \frac{G'(\varphi)}{\varphi}\chi_{\{\varphi\neq 0\}} \varphi$, with $\chi_A$ being the characteristic function of the set $A$.
 Since $\varphi$ is nontrivial and $V\star G(|\varphi|)  \frac{G'(\varphi)}{\varphi}\chi_{\{\varphi\neq 0\}}\in L^\infty_{loc}(\R) $, we conclude that 
$\varphi>0$ in $\R$ by the Harnack inequality (see Lemma 4.9 in \cite{Cabre17}) and a standard argument of intersecting balls.\\
 
{\bf Case 1} : $u \equiv 0$

{\bf Case 2} : $v \equiv 0$

{\bf Case 3} : $u \neq 0 $ and $v \neq 0$ everywhere.\\

Then \eqref{eqq13} implies that
\begin{align*}
&\frac{u(x)}{v(x)} =\frac{u(y)}{v(y)}\;\quad \forall\; x,y \in \mathbb{N}^N,\\
&\Rightarrow \frac{u(x)}{v(x)} = \alpha \Rightarrow u(x) = \alpha v(x)\quad\forall\;
x \in \mathbb{R}^N,\\
&z = (\alpha + i) v\Rightarrow z = e^{i\sigma}w,w = |z|.
\end{align*}
Let us now prove \eqref{eqq13}. By the fact that $\mathcal{J}(z) =
 \hat{\mathcal I}_\lambda$, we can find a Lagrange multiplier
 $\alpha \in \mathbb{C}$ such that $\mathcal{J}'(z)(\xi)=
 \displaystyle{\frac{\alpha}{2} \int_{\mathbb{R}^N}}z \bar{\xi}
 + \xi\bar{z}$ for all $ \xi \in H^s$.
Putting $\xi = z$, it follows immediately that $\alpha \in
\mathbb{R}$ and
$$\label{eqq14}\left\lbrace
\begin{array}{l}
 \displaystyle{\int}_{\R}\nabla_s u \nabla_s f -\displaystyle{\int}_{\R\times\R}
 G(u^2+v^2)^{1/2}(y)V(|x-y|)dyG'(f(x))dx \\ \\\hskip215pt=\alpha
 \displaystyle{\int}_{\R} u(x)f(x)dx, \\ \\
\displaystyle{\int}_{\R}\nabla_s v\nabla_s f -
\displaystyle{\int}_{\R\times\R}G(u^2+v^2)^{1/2}(y)V(|x-y|)dy
G'(f(x))dx \\ \\ \hskip215pt= \alpha
\displaystyle{\int}_{\R}v(x)f(x)dx,
 \end{array}\right.
$$
$\nabla_s$ denotes the fractional gradient, for all $f \in H^s$. It
follows that $u$ and $v$ solve the following system
$$\left\{ \begin{array}{ll}
(-\Delta)^s\,u + \displaystyle{\int}G(u^2+v^2)^{1/2}(y)V(|x-y|)dy\,
G'(u(x)) + \alpha u(x) = 0,\\
(-\Delta )^s\,v + \displaystyle{\int}G(u^2+v^2)^{1/2}(y)V(|x-y|)dy\,
G'(v(x)) + \alpha v(x) = 0.
\end{array}\right.$$
By Proposition \ref{prop:reg}, we have that  $u$ and $v \in C(\mathbb{R}^N)$ because $(u^2+v^2)^{1/2}\in \s $.
 Let $\Omega = \{x \in \mathbb{R}^N : u(x) = 0\}$, obviously $\Omega$ is closed since $u$ is continuous. Let us prove that it is also open. Suppose that $x_0 \in \Omega$. Knowing that $\varphi(x_0) > 0$, we can find a ball $B$ centered in $x_0$ such that $v(x) \neq 0$ for any $x \in B$. Replacing $u$ and $v$ in \eqref{eqq8}, we certainly have that
\[u(x) v(y) - v(x) u(y) = 0\quad \forall\; x, y \in B.\]
This proves the result.
\end{proof}
\section*{Appendix}
\noindent In this appendix, we prove the following 
\begin{proposition}\label{prop:reg}
Let $s\in(0,1), N-2s\leq \beta<N, \beta>0, u,\varphi\in \s$, $G$ such that $\mathcal{A}_0$ holds and $\kappa$ is a real number  such that 
\begin{equation}(-\Delta)^s u -\kappa u = 
 [V\star G(\varphi)] G'(u).\label{theeq}
\end{equation}
Then, there exists $\alpha\in (0,1)$  depending only on $N,\kappa,s,\beta$ such that $u\in C^{0,\alpha}_{loc}(\R)$. 
Moreover, if $\varphi \in L^\infty_{loc}(\R)$, then $u\in C^{0,\alpha}_{loc}(\R)$ if $\beta\leq 1$
 and $u\in C^{1,\alpha}_{loc}(\R)$ if $\beta>1$ and in addition $V\star G(\varphi) \in C^{0,\alpha}_{loc}(\R)$.
\end{proposition}
\begin{proof}
We start by recalling the   Gagliardo-Nirenberg  inequality 
$$
 \nr\varphi\nr_{L^p(\R)} \leq c_{N,s,p} \nr\varphi\nr_{H^s(\R)}\quad \textrm{   for all $\varphi\in \s$   },
$$ 
for $p\in \left[2, \frac{2N}{N-2s}\right]$ 
if $N>2s$ and for all $p\in \left[\left.2, \frac{2N}{N-2s}\right)\right.$ and $2s\geq N$ (here we put $ \frac{2N}{N-2s}\equiv+\infty$). Also we recall the  Hardy-Littlewood-Sobolev inequality:
$$
 \|V\star g\|_{L^{\frac{qN}{N-q\beta}}(\R)} \leq C_{N,\beta,q}  \|  g\|_{L^{q}(\R)} \quad\textrm{ for every $g\in L^{q}(\R)$,} $$ 
for $ N-q\beta>0$.
\vskip6pt
First of all we focus on the case $N>2s$. Thus, we have 
\[
\nr G(\varphi)\nr_{L^q(\R)}\leq
\nr\varphi^2\nr_{L^q(\R)}+\nr|\varphi|^\mu\nr_{L^q(\R)}=
\nr\varphi\nr_{L^{2q}(\R)}^2+\nr\varphi\nr_{L^{\mu q}(\R)}^\mu.
\]
Hence, since $\varphi\in\s$, we infer that $G(\varphi) \in L^q(\R)$ provided that $1\leq q\leq \frac{N}{N-2s}$ and $\frac2\mu\leq q\leq \frac1\mu\frac{2N}{N-2s}$, that is  $1\leq q\leq \frac{N}{N-2s}$ and $1\leq q \leq \frac{2N^2}{(N-2s)(N+2s+\beta)}$. Now, thanks to the fact that $N-2s\leq\beta<N$, we get $1<\frac N\beta\leq \frac{N}{N-2s}$ and $1<\frac N\beta \leq\frac{2N^2}{(N-2s)(N+2s+\beta)}$. In particular, we deduce that $G(\varphi) \in L^q(\R)$ for all $q\in \left[1,\frac N\beta\right]$. Now, for all $\epsilon>0$ we let $q_\epsilon =\frac N\beta-\epsilon>1$. Using  the Hardy-Littelwood-Sobolev inequality, we get $V\star G(\varphi) \in L^{\frac{N q_\epsilon}{\epsilon\beta}}(\R)$ which in turns with the fact that $\beta\geq N-2s$ shows that $V\star G(\varphi) \in L^{r}(\R)$ for all $r>\frac{N}{N-\beta}\geq \frac{N}{2s}$. Now, using the notation $b(x)=\frac{G'(u)}{1+|u|}$ and $sign(u)=\frac{u}{|u|}$, we reformulate the equation \eqref{theeq} as follows 
\begin{eqnarray*}(-\Delta)^s u(x) -\kappa u(x) &=& [V\star G(\varphi)] \,b(x)\,(1+|u|)),\\&=&\int_{\R} V(|x-y|)G(\varphi(y))dy \,b(x)\,(1+sign(u)\,u).
\end{eqnarray*}
Observing that $\mu-2<\frac{2N}{N-2s}-2=\frac{4s}{N-2s}$, then for all $r>\frac{N}{2s}$, we can write 
\begin{align*}
\nr [V\star G(\varphi)]b \nr_{L^r(\R)} &= \nr [V\star G(\varphi)]\, \frac{G'(u)}{1+|u|}\nr_{L^r(\R)}, \\
 &\leq  c\,\nr [V\star G(\varphi)]\frac{|u|+|u|^{\mu-1}}{1+|u|}\nr_{L^r(\R)}, \\
 &\leq c\,\nr V\star G(\varphi)\nr_{L^r(\R)} + c\,\nr [V\star G(\varphi)]|u|^{\mu-2}\nr_{L^r(\R)}.
\end{align*}
In order to deduce that the right hand side of this estimate is finite, we use H\"older's inequality to get
\begin{align*}
\nr [V\star G(\varphi)]\,|u|^{\mu-2}\nr^r_{L^r(\R)} &\leq\nr V\star G(\varphi)]\nr_{L^\frac{r\,\theta}{\theta-1}(\R)}\,\nr |u|\nr^{\mu-2}_{L^{r\,(\mu-2)\,\theta}(\R)},
\end{align*}
 for all $\theta>1$. Therefore, we can choose $r>\frac{N}{2s}$ and $\theta>1$ respectively close to $\frac{N}{2s}$ and $1$ so that $1<r\,(\mu-2)\,\theta<\frac{2N}{N-2s}$.
Hence  using the Gagliardo-Nirenberg inequality   and the fact that $V\star G(\varphi) \in L^{r}(\R)$ for all $r>\frac{N}{N-\beta}\geq \frac{N}{2s}$ and $u\in \s$, 
we end up with 
$ [V\star G(\varphi)]b \in L^r$ for some $r>\frac{N}{2s}$, hence $u\in C^{0,\alpha}_{loc}(\R)$.
\vskip6pt
Now, we write the equation \eqref{theeq} as follows
\begin{align*}&(-\Delta)^s u(x) = c(x) u(x)+ d(x):\\
&c(x)=\kappa+[V\star G(\varphi)]\,b(x)\,sign(u) \in L^r(\R),\\
&d(x)=[V\star G(u)]\,b(x)\in L^r(\R),
\end{align*}
for some $r>\frac{N}{2s}$. Thus, using the regularity result of Ref. \cite{TX}, 
we conclude that $u\in C^{0,\alpha}_{loc}(\R)$ for some $\alpha\in(0,1)$ provided $\frac{N}{2s}>1$.
 If $N=1$ and $s>\frac12$, then it is well-known that $\s$ is embedded in $ C^{0,\alpha}_{loc}(\R)$ with $\alpha=s-\frac12-\left[s-\frac12\right]$ 
so that $u\in C^{0,\alpha}_{loc}(\R)$. Moreover, if $N=1$ and $s=\frac12$,
 we have obviously $u\in L^p(\R)$ for every $p\geq 2$ and classical elliptic regularity yields $u\in C^{0,\alpha}_{loc}(\R)$ for some $\alpha\in(0,1)$.
\vskip6pt
 In the following, $[\cdot]$ stands for the integer part of $\cdot$.
Let us introduce a cutoff function $\eta\in C_c^{\infty}(\R)$ such that $ \eta\equiv 1$ in the closed ball $B_R$ of center $0$ and radius $R>0$
 and $\eta\equiv 0$ in $\R\setminus B_{2R}$. To alleviate the notation, we denote   $f=G(\varphi)$ which belongs to $L^\infty_{loc}(\R)\cap L^{q}(\R)$ with $1<q\leq \frac N\beta$.
We define   $V_1(\varphi):= V \star (\eta f )$ and   $V_2(\varphi):= V \star ((1-\eta) f )$. Then using Fourier transform, we get  
  $(-\Delta)^{\frac\beta2}V_1(\varphi)=f$ in the sense of distributions.
 Now, if $\frac\beta2 \in \mathbb{N}^\star$, then it is rather easy to show using classical regularity theory that $V\star G(\varphi) \in C^\beta(\R)$. 
Next, if $ 0<\frac\beta2<1$, then we apply Proposition 2.1.9 of Ref. \cite{Sil}
 to show that $V_1(\varphi) \in C^{0,\alpha}(\R)$ for $\beta\leq 1$ and $V_1(\varphi) \in C^{[\beta],\alpha}(\R)$ for $\beta>1$ and some $\alpha\in(0,1)$.
 Now, $V_2(\varphi)$ is smooth on $B_R$ since it is $\frac\beta2-$harmonic in such a ball, see Ref. \cite{BB-Schr}.
 Hence, $V\star G(\varphi) \in  C^{0,\alpha}_{loc}(\R)$ for $\beta\leq 1$ and $V\star G(\varphi) \in C^{[\beta],\alpha}_{loc}(\R)$ for $\beta>1$ and some $\alpha \in (0,1)$. Let us now turn to the case of $\frac\beta2>1$ and $\frac\beta2\not\in \mathbb{N}$. we let $\sigma=\frac\beta2-\left[\frac\beta2\right]$. Using Fourier transform, we have 
\[(-\Delta)^{\left[\frac\beta2\right]}V_1(\varphi)=(-\Delta)^{\left[\frac\beta2\right]}\left((-\Delta)^{\sigma} V_1(\varphi)\right)=\eta\,f\]
in the sense of distributions. Again, classical regularity theory arguments implies that $(-\Delta)^{\sigma} V_1(\varphi)\in C^{[\beta]}(\R)$ and so $V_1(\varphi) \in C^{[\beta]}(\R)$. Similarly, we have 
\[(-\Delta)^{\frac\beta2}V_2(\varphi)=(-\Delta)^{\sigma}\left((-\Delta)^{\left[\frac\beta2\right]} V_2(\varphi)\right)=(1-\eta)\,f\]
in the sense of distributions. Therefore the function $g:=(-\Delta)^{\left[\frac\beta2\right]}V_2(\varphi)$ is given by 
\[(-\Delta)^{\left[\frac\beta2\right]}V_2(\varphi)(x)=\int_{\R}\frac{(1-\eta(y))\,f(y)}{|x-y|^{N-\sigma}}\,dy.\]
Also, using the Hardy-Littelwood-Sobolev inequality, it is rather straightforward to see that $g\in L^p(\R)$ 
for some $p>1$. Thus, $g$ belongs to the set $\left\lbrace u,\:\int_{\R}\frac{|u(x)|}{1+|x|^{N+2\sigma}}\,dx<+\infty   \right\rbrace$.
 Again, since $g$ is $\sigma-$harmonic in $B_R$, we deduce that $g$ is smooth on $B_R$ by Ref. \cite{BB-Schr}. 
The radius $R$ being arbitrary, it follows that $V_2(\varphi)$ is smooth on $\R$. 
In particular, we have $V_2(\varphi) \in C^{[\beta]}(\R)$ because $\left[\frac\beta2\right]$ is a positive integer. 
Recalling that we showed $V_1(\varphi) \in C^{[\beta]}(\R)$, we conclude $V\star G(\varphi) \in C^{[\beta]}(\R)$.
\vskip6pt
Let us now summarize and conclude the proof. We considered the partial differential equation \eqref{theeq}
 and proved that for some $\alpha \in(0,1)$, we have  $V\star G(\varphi) \in C^{0,\alpha}_{loc}(\R)$ 
for $\beta\leq 1$ and $V\star G(\varphi) \in C^{[\beta],\alpha}_{loc}(\R)$ for $\beta>1$. Since $G'$ is
 locally Lipschitz, we deduce that $u\in C^{0,\alpha}_{loc}(\R)$ for $\beta\leq1$ and $u\in C^{1,\alpha}_{loc}(\R)$ for $\beta>1$ 
by adapting the proof of Lemma 3.3 of Ref. \cite{Fall-Felli} for $N>2s$. If $N=1$ and $2s\geq 1$, we have that $[V\star G(\varphi)]\,G'(u) \in C^{0,\gamma}_{loc}(\R)$ for some $\gamma \in (0,1)$, thus using Proposition 2.1.8 of Ref. \cite{Sil}, we get $u\in C^{1,\alpha}_{loc}(\R)$.
\end{proof}

\section*{Acknowledgments} 
Y. Cho was supported by NRF grant 2010-0007550 (Republic of Korea).
 M. M. Fall is supported by the Alexander von Humboldt foundation.


\begin{thebibliography}{00}
\bibitem{BB-Schr} K. Bogdan, T. Byczkowski, {\it Potential theory for the $\alpha$-stable Schr\"{o}dinger operator on bounded Lipschitz domains}, Studia Math. 133 (1999), no. 1, 53-92.
\bibitem{Cabre17}X. Cabre, Y Sire, {\it Nonlinear equations for fractional Laplacians I: Regularity, maximum principles, and Hamiltonian estimates}. 
 Annales de l'Institut Henri Poincar\'e (C)
 Non Linear Analysis, to appear.      
\bibitem{caz} T. Cazenave,  {\em Semilinear Schr\"{o}dinger equations},
Courant Lecture Notes in Mathematics, 10. New York University,
Courant Institute of Mathematical Sciences, New York; American
Mathematical Society, Providence, RI, 2003.
\bibitem{Lions1} T. Cazenave, P.-L. Lions, {\it Orbital stability of standing waves for some nonlinear Schršdinger equations},  Comm. Math. Phys. Volume 85, Number 4 (1982), 549-561.

\bibitem{chho}  Y. Cho, H. Hajaiej, G. Hwang, T. Ozawa, {\it On the Cauchy problem of fractional Schr\"{o}dinger equation with Hartree type nonlinearity} to appear in Funkcialaj Ekvacioj (arXiv:1209.5899).
\bibitem{chhwoz} Y. Cho, G. Hwang, T. Ozawa, {\it Global well-posedness of critical nonlinear Schr\"{o}dinger equations below
$L^2$}, DCDS-A 33 (2013), 1389-1405.
\bibitem{chonak} Y. Cho, K. Nakanishi, {\it On the global existence of semirelativistic Hartree equations}, RIMS Kokyuroku Bessatsu, B22 (2010), 145-166.
\bibitem{Fall-Felli} M. M. Fall and V. Felli,
 {\it Unique continuation property and local asymptotics of solutions to fractional elliptic
 equations}, http://arxiv.org/abs/1301.5119. Comm. Partial Differential Equations, to appear.
\bibitem{GH} B. Guo, D. Huang, {\it Existence and stability of
standing waves for nonlinear fractional Schr\"odinger equations},
J. Math. Phys. 53 (2012). 083702
\bibitem{Xi} X. Guo, M. Xu, {\it Existence of the global smooth
solution to the period boundary value problem of fractional nonlinear
Schrodinger equation}, J. Math. Phys. 47, 082104 (2006).
\bibitem{Hichem} H. Hajaiej, {\it Existence of minimizers of functionals involving the fractional gradient in the absence of compactness, symmetry and monotonicity}, J. Math. Anal. Appl. 399 (2013) 17-26.
\bibitem{hajaiej13} H Hajaiej, L. Molinet, T. Ozawa, B. Wan,, {\it Necessary and Sufficient Conditions for the Fractional Gagliardo-Nirenberg Inequalities and Applications to Navier-Stokes and Generalized Boson Equations},  RIMS Kokyuroku Bessatsu   RIMS Kokyuroku Bessatsu B26, 159-175, 2011-05-Kyoto University
\bibitem{hajaiej14} H. Hajaiej, C. A. Stuart, {\it On the Variational Approach to the Stability of Standing Waves for the Nonlinear Schr\"odinger
Equation}, Advanced Nonlinear Studies 4 (2004), 469-501.
\bibitem{Laskin1} N. Laskin, {\it Fractional quantum mechanics and L\'evy integral}, Phys. Lett. A 268, 298305 (2000).
\bibitem{Laskin2} N. Laskin, {\it Fractional quantum mechanics}, Phys. Rev. E 62, 3135 (2000).
\bibitem{Laskin3} N. Laskin, {\it Fractional Schrodinger equations}, Phys. Rev. E 66, 056108 (2002).
\bibitem{Lions} P.-L. Lions, {\it The concentration-compactness method in the calculus of variations. The locally compact case. Part I}, Ann. Inst. H. pincar\'e Anal. Non Lin\'eaire 1 (1984) 109--145; Part II, Ann. Inst. H. pincar\'e Anal. Non Lin\'eaire 1 (1984) 223--283
\bibitem{Lieb15} E.H. Lieb, H. T. Yau, {\it The Chandrasekhar theory of stellar collapse as the limit of quantum mechanics}, Commun. Math. Phys., 112 (1987), 147-174.
\bibitem{Sil}L. Silvestre, {\it Regularity of the obstacle problem for a fractional power of the
   Laplace operator}, Comm. Pure Appl. Math., 60(1):67-112, 2007.
   \bibitem{TX} J. Tan; J. Xiong, {\it A Harnack inequality for fractional Laplace
equations with lower order terms}, Discrete Contin. Dyn. Syst. 31
(2011), no. 3, 975–983.
\bibitem{Wu}  D. Wu, {\it Existence and stability of standing waves for nonlinear fractional Schr\"{o}dinger equations with Hartree type nonlinearity}, arXiv:1210.3887
\end{thebibliography}
\end{document}